\documentclass[12pt]{amsart}
\usepackage{latexsym,amssymb,amsmath,amsthm,curves}
\usepackage{tikz,epic}

\newcommand{\old}[1]{}
\theoremstyle{plain}
\newtheorem{thm}{Theorem}[section]
\newtheorem{lem}[thm]{Lemma}
\newtheorem{conj}{Conjecture}
\newtheorem{cor}[thm]{Corollary}

\newtheorem{prop}[thm]{Proposition}
\theoremstyle{definition}
\newtheorem{defn}[thm]{Definition}

\newtheorem{ex}[thm]{Example}
\newtheorem{rk}[thm]{Remark}
\newtheorem{qn}[thm]{Question}

\def\reals{{\mathbb R}}

\def\bc{{\bf c}}

\title[1-skeleta of polytopes as posets]{Posets arising as 1-skeleta of simple polytopes, the nonrevisiting path conjecture, and poset topology}

\author{Patricia Hersh} %$^*$}
\address{Department of Mathematics, University of Oregon, Eugene, OR 97403}
\email{plhersh@uoregon.edu}

\thanks{
 This work was supported by 
NSF grants DMS-1200730, DMS-1500987, DMS-1953931 and NSF conference grant DMS-1101740.
}


\begin{document}

\begin{abstract}%%
Given any polytope $P$ and any  generic linear functional ${\bf c} $, one obtains  a directed graph $G(P,\bc)$  from  the 1-skeleton of $P$  by orienting each edge $e(u,v)$ from $u$ to $v$ for $\bc \cdot u <\bc  \cdot v$.    %This paper studies simple polytopes $P$ in which $G(P,\bc )$ is the Hasse diagram of a poset or a lattice.   
For $P$  a simple polytope  and $G(P,\bc )$  the Hasse diagram of a lattice $L$, 
 the join of any collection $S$  of elements which  all cover a common element $u$  in $L$  is proven to equal the sink of the smallest face  of $P$ containing $u$ and all of the  elements of $S$.  
The author conjectures % something stronger 
for such  $G(P,\bc )$ %, namely 
that no directed path in $G(P,\bc )$ ever revisits any facet of $P$.  % which it has left.  
This   %conjecture  
would imply for such $P$ and $\bc $ that the simplex method for linear programming is efficient under all possible  pivot rules.   This conjecture is proven for 3-polytopes and  for spindles. % (the source of   
% having as source and sink  the two vertices $v,w$  of the spindle with the property that each facet  of the spindle contains  $v$ or $w$; %one of these two vertices; 
%This result for spindles
%this result for spindles shows that 
%the known counterexamples to the Hirsch Conjecture). % do not provide counterexamples to the author's  conjecture.   

For simple polytopes in which $G(P,\bc)$  is the Hasse diagram of a  lattice  $L$, 
the order complex of each open interval in  $L$  is proven  %to be  
 homotopy equivalent to a ball or a sphere. % of  some dimension.  
  Applications  are given to the 
weak Bruhat order,  the Tamari lattice, and  %more generally to  
the
Cambrian lattices.
%, with their Hasse diagrams arising from 1-skeleta of  
%permutahedra, associahedra, and 
%generalized associahedra, respectively.
 
This paper concludes with  an 
appendix by Dominik Preu\ss  \hspace{.02in} proving the 
monotone Hirsch conjecture for  $P$ a simple polytope and   $G(P,\bc )$ the Hasse diagram of a lattice.  This  confirms one of the main consequences that  
 the author's   conjecture would have. 
 \\
 Data availability statement: Data sharing not applicable to this article as no datasets were generated or analysed during the current study.
\end{abstract}

\maketitle

\section{Introduction}\label{intro-section}

This paper undertakes a study of directed graphs $G(P,\bc )$ which are 
%simultaneously the Hasse diagrams % (see Section ~\ref{disc-morse-bg-section} for terminology review)  
%of partially ordered sets  %$L$ 
%and 
the 1-skeleta of simple polytopes $P$  %in 
in $\reals^d $ with each edge $e_{u,v}$ of the polytope  
oriented in the direction that ``cost''    increases, focusing on the case when $G(P,\bc)$ is also the Hasse diagram of a partially ordered set (poset). 
 The {\bf cost} of a vertex $v$ in a polytope $P$ is  the dot product
$\bc \cdot v$ where $\bc $ is a fixed vector in $\reals^d$ known as the cost vector.    
%what  we mean when we say ``cost increases'' is 
%that the dot product $\bc \cdot u$  is smaller than the dot product $\bc \cdot v$ for the  given cost vector 
%$\bc \in \reals^d$.   % (see Section ~\ref{disc-morse-bg-section} for terminology review).  
%Some prominent examples include the weak order (with Hasse diagram  the 1-skeleton of the permutahedron), the Tamari lattice (with Hasse diagram the 1-skeleton of the associahedron) and the Cambrian lattices (with Hasse diagrams the 1-skeleta of the generalized associahedra).    
A major reason there has been as much interest  as there has been over the last several decades % as there has been  
in  %understanding 
the graphs $G(P,\bc )$ is the following connection to the simplex method for linear programming.   The task of 
linear programming  is  to 
find  the point $\bf{v} $  in a polytope (or polyhedron or more general convex body) where the cost $\bc \cdot \bf{v}$  is maximized (or minimized); the simplex method accomplishes this by starting at a vertex  of $P$ and  moving greedily  along the edges of  the 1-skeleton of $P$  
% $G(P,\bc )$
 in a manner that increases   (resp. decreases)
  cost at each step, or  in other words along the directed edges of $G(P,\bc )$, 
 until reaching the sink (resp. source) of $G(P,\bc )$.   For this reason, there is    
particular interest in  %the question of giving 
 upper bounds on the % directed diameter of $G(P,\bc )$  %
  %length of the 
  diameter of $G(P,\bc)$ and on the length of the   longest  path %``monotone path'' 
 in $G(P,\bc )$.  % (i.e. the longest ``monotone path''). 
 %, the latter of which
 % is the longest  path in the 1-skeleton of $P$ upon which cost is monotonically increasing (or equivalently  is monotonically decreasing).  
 %, namely the longest sequence 
%$v_1\rightarrow v_2\rightarrow \cdots \rightarrow v_k$ 
%of vertices of $P$ where each consecutive pair of vertices is connected by an edge and where the cost monotonically increases (resp. decreases) at each step.  

An important  approach to proving such upper bounds  was a conjecture known as the nonrevisiting path conjecture.  This posited  the existence of a path from any vertex  $u$ to any vertex  $v$ of $P$  in the undirected version of the graph  $G(P,\bc )$ with the property that the path  would not revisit any facet (i.e. maximal face in the boundary of $P$)  from which it had departed.    
%%% By this we mean that the path cannot include a series of consecutive  vertices $v_1\rightarrow v_2\rightarrow \cdots \rightarrow v_k$ where $v_1$ and $v_k$ are both vertices contained in a facet $F$ of $P$ while $v_2$ is  not contained in $F$.  
While this conjecture was proven false by Santos in \cite{Sa}, there is still interest in finding  classes of polytopes and cost vectors for which it does hold.  We conjecture %(in Conjecture ~\ref{main-conj}) 
for  $P$ a simple polytope and $G(P,\bc )$ the Hasse diagram of a lattice that something % in some sense 
  far stronger than the nonrevisiting path conjecture should be true.  
  %That is, we conjecture that  every directed path in $G(P,\bc )$ has  the property that it never revisits any facets.  
  See Conjecture ~\ref{main-conj}.
We prove a number of results regarding the graphs $G(P,\bc )$ under the assumptions that $P$ is a simple polytope and that 
$G(P,\bc )$ is the Hasse diagram of either  a 
poset or more specifically of a lattice.  We are hopeful that these results may be useful  steps towards proving Conjecture ~\ref{main-conj}.   
%The two main results in the paper, both described  later in this section, are Theorems  ~\ref{pseudo-is-join} and ~\ref{thm-2}.
Our starting point for this work is a pair of observations, the former of which is well-known and the latter of which is proven later in the paper: %in Lemma ~\ref{1-d-revisiting}:
\begin{enumerate}
\item
Requiring ``monotone paths''  %(i.e. paths where cost increases at each step)
 in the 1-skeleton of a polytope  (namely paths where cost increases at each step)
never to revisit any facets (a property we call the {\bf nonrevisiting property})  % that they have left % in a polytope $P$ 
implies that they also cannot revisit any faces of any dimension,  % that they have left, 
by virtue of each face in a polytope being a finite intersection of facets.
\item
Requiring monotone paths in the 1-skeleton of a polytope $P$ never to revisit any 1-dimensional faces % in $P$ 
%that they have left
is equivalent to requiring $G(P,\bc )$ to be the Hasse diagram of a finite partially ordered set (poset).  This will be proven in Lemma ~\ref{1-d-revisiting}.
\end{enumerate}

Combining these observations, notice that  if $G(P,\bc )$ has  the nonrevisiting property,
% property
%for each of its monotone paths 
% (which we call the {\bf nonrevisiting property})  that  its monotone paths  
%that the path 
 %may never revisit any facet, % (a property we call the {\bf nonrevisiting property}),
% it has departed,   then 
%it is a necessary condition that 
%it is necessary that 
then $G(P,\bc )$ must  
be the Hasse diagram of a poset  (in which case we say that $G(P,\bc )$ has the {\bf Hasse diagram property}).   
% We call the former  property the {\bf nonrevisiting property} and the latter the {\bf Hasse diagram property}.   
%\end{enumerate}
%; to be precise what we mean be this, we say that a directed visits a face when it visits any vertex in that face.
%
%Bootstrapping on
% Equipped with 
 Spurred  on by these observations, % Taking advantage of this observation, 
this paper  combines poset theoretic techniques with ideas from discrete geometry  to prove structural results regarding  $G(P,\bc )$  in the case when $P$ is a simple polytope (see Section ~\ref{disc-morse-bg-section} for the definition)  and  $G(P,\bc )$ is the Hasse diagram of a poset or more specifically of a lattice (also  reviewed in 
Section ~\ref{disc-morse-bg-section}).  
% These two observations above together with our upcoming structural results regarding $G(P,\bc )$, discussed shortly, will lead us to  
%We also conjecture conditions on $P$ and $\bc $ which we believe will  imply that the monotone paths in $G(P,\bc )$ cannot revisit any facets. 
% that they have left % suffice  for the aforementioned face nonrevisiting property for directed paths in $G(P,\bc )$  to hold 
%This appears as  Conjecture ~\ref{main-conj}. %; our results include verifying this conjecture in special cases as well as several results that we believe 
%could be steps towards a full proof of this conjecture.  

  One motivation  for  studying polytopes with this  %sort of face 
  nonrevisiting property is that it   %property  
  implies an upper bound of $n-d$ % on the diameter and  
on the  directed diameter of %directed graph structure  
$G(P,\bc)$  %on the  
%1-skeleton of the polytope $P$  
where $n$ is the number of facets in $P$ and $d$  is the dimension of $P$.   In fact it  %implies more: it 
% Such a %  diameter 
%bound  % on $G(P,\bc )$ 
%has ramifications for a central question of operations research, namely the question of upper bounding the complexity of the simplex method for linear programming: 
%, though it remains to be seen whether this work will have any practical implications in this realm:
 %the sort of face non-revisiting property for facets in a polytope $P$ with respect to a cost vector $\bc $ that we focus upon in this paper  
 forces the simplex method for  linear programming to run on $P$ with cost vector $\bc $ 
 in at most $n-d$ steps 
 % with  
 %respect to
 %the 
  %cost vector $\bc $, %specifically  to run in at most $n-d$ steps, % with respect to cost vector $\bc $ 
 regardless of choice of pivot rule (another notion 
 reviewed in Section ~\ref{disc-morse-bg-section}). % (another notion that is reviewed  in Section ~\ref{disc-morse-bg-section}). 
 %; more specifically,  the simplex method  runs  in at most $n-d$ steps no matter what pivot rule is used.   
 % A main goal of our paper is to begin to understand what conditions on $P$ and $\bc $ will force this sort of face nonrevisiting property.
%
%, a statement which directly translates to an upper bound  of $n-d$ on the running time for the simplex method for linear programming for such $P$ and $\bc$.   
%In particular, this  implies   the Hirsch Conjecture for any $P$ admitting such a cost vector $\bc $.  
A second, quite different type of  motivation for our work  
is that lattices  whose Hasse diagrams may be realized as 1-skeleta of simple polytopes will turn out to have well-controlled  topological structure, as will be proven  in Theorem ~\ref{thm-2}; this gives a new and unified  explanation why several classes of posets all  have  all of their  open  intervals homotopy equivalent to balls or spheres.   Turning this around, Theorem ~\ref{thm-2} % (as well as other  structural results   in this paper) 
may give some new insight into the question of which Hasse diagrams of posets may be realized as 1-skeleta of (simple) polytopes.  See  Remark ~\ref{which-Hasse} for more on this. 
The famous Klee-Minty cubes (introduced in  \cite{KM} and discussed more  in 
Example ~\ref{klee-minty})  violate  not only our nonrevisiting property, but also the Hasse
diagram property.  % requirement that  $G(P,\bc )$ be the Hasse diagram of a poset, 
In fact, they violate this  in a way that really seems to be at the heart  of why the simplex method for linear programming can be 
so inefficient on Klee-Minty cubes.  % See Example ~\ref{klee-minty} for a more in-depth discussion of this important example.
%One could view this as a piece of evidence for Conjecture ~\ref{main-conj}.  

   We  assume  throughout this paper  that $\bc $ is   %a simple polytope $P \subseteq \reals^d$ and 
  a ``generic''  cost vector in  $\reals^d$, by which we mean  that $\bc\cdot u \ne \bc\cdot v$ for each pair $u,v$ of vertices of our polytope $P\subseteq \reals^d$ that are the two endpoints of an edge in $P$.   %Given such a vector $\bc $, we   obtain  a  directed graph $G(P,\bc )$  on  the 1-skeleton of $P$ by orienting each edge $e_{u,v}$ from $u$ to $v$ for $\bc\cdot u < \bc\cdot v$.   
  The resulting directed graph $G(P,\bc )$
   %on the 1-skeleton of $P$ with edges oriented in the direction that the ``cost'' $\bc \cdot v$ of a vertex $v$ increases 
    is easily seen to be % an 
    acyclic.  % directed graph. %, by which we mean that this graph  will not  have  any directed cycles.  
  %We study the case in which $G(P,\bc )$  is a Hasse diagram of a partially ordered set (poset) or more specifically of a lattice.    
  While others have previously studied combinatorial questions related to the simplex method  for linear programming and  to diameter bounds on polytopes (see e.g.  \cite{Sa}, \cite{Todd}, \cite{KM}, \cite{KW}, \cite{BDL}), %, \cite{deloera}), 
we are not aware of any other work in which $G(P,\bc )$ is assumed to be a Hasse diagram.  This condition  is not only necessary to have the nonrevisiting  property,
% of %directed paths not revisiting any facets,  
 but  is  also a surprisingly  useful input for  many of our proofs. %n assumption which we believe to be quite helpful 
%for  proving  particular classes of polytopes and cost vectors  have  this  facet nonrevisiting property.  %terms of forcing 
% case which we believe may be  the best possible scenario in terms of 
%  all possible directed paths in $G(P,\bc )$  not to revisit any facets that they have left.
  % having the property of not revisiting faces.   %  A natural follow-up question to our work is to ask to what extent our hypotheses may be relaxed while still preserving some of the structure leading t
%; in particular, we do not know of any other work combining poset theoretic techniques with polytopal methods.  

%Remark
%~\ref{simple-enough} explains how  in some sense 
% o
 %One might  also argue by the following reasoning  that  o
% In some sense, our focus on simple polytopes is not such a severe 
%restriction in  seeking a   better understanding of  which polytopes will satisfy strong 
%upper bounds on their diameters.   After all,
Our  focus  on simple polytopes throughout this paper is in some sense not such a severe restriction, since Klee and Walkup proved in \cite{KW} that the 
Hirsch Conjecture (reviewed in Section ~\ref{disc-morse-bg-section}) for simple polytopes would have implied the Hirsch Conjecture  for all polytopes.
%, showing that in some sense this is not such a severe restriction for us to make.   
%It seems plausible 
%that  s
Many  of our results might also hold  for non-simple  polytopes, with  obvious small modifications to the statements of the results.   Our proofs
do  heavily rely upon our assumptions  % the power when taken together of  %our  %hypothesis that  
%the fact 
that our polytopes $P$  are  %assumed to be  
simple and % as well as our assumption  
that  the graphs %, particularly in conjunction with our assumption that %  on  %our hypothesis  that 
$G(P,\bc )$ are Hasse diagrams.  
%We have found this pair of hypotheses to be  quite powerful  especially   in conjunction with each other.  %when taken together.   

Now let us describe our main results, beginning with some terminology that this will require.  %, beginning with notions that these will require.  
  %We  also  make progress towards answering this question.  
 % The main focus of this paper is to prove a series of structural results  regarding these graphs $G(P,\bc )$ 
 % Our main result is a  structural result which we now describe, starting with terminology we will need.   
%  a result that we believe constitutes  a substantial step towards proving Conjecture \ref{main-conj}.
  % that for $P$ a simple polytope and $G(P,\bc )$  the Hasse diagram of a lattice $L$, the join of any collection of atoms of any interval $[u,v]$ in $L$ equals the sink of the smallest face of the polytope containing all these atoms.   
 Given a  polytope $P$ and a generic cost vector $\bc $, each face $F$ of $P$ has  a unique  source  and a unique sink in the restriction of $G(P,\bc )$ to $F$.
 %, namely a vertex $v$  whose directed edges to other vertices of $F$ all point outward from $v$. 
%  Each  face $F$  also will have a unique  sink. %, namely a vertex $w$  whose edges to other vertices of $F$ all point inward to $w$.   
Given % an $i$-dimensional 
%face $F$ (usually called an $i$-face)
%of 
a simple polytope $P$  such that $G(P,\bc )$ is the Hasse diagram of a poset $Q$, then for any  $u\in Q$ and any collection $a_1,\dots ,a_i$ of elements all covering
$u$ in $Q$, there is a unique % $i$-dimensional face 
% the source vertex of the unique  
$i$-dimensional face $F$ in $P$  containing $u$ along with  the outward edges from $u$ to the elements $a_1,\dots ,a_i$.  %\in F$ all covering $u$ in $Q$,  
We then define  the  {\bf pseudo-join}   of $a_1,a_2,\dots ,a_i$ to be   the unique sink of this face 
$F$.  %  The validity of this notion is
This definition is further explained and  justified in Section ~\ref{def-section}.
%Definition ~\ref{pseudo-def} for the precise definition of pseudo-join.
One might hope in the case that $Q$ is a lattice  that the pseudo-join of $a_1,\dots ,a_i$ 
% this  geometric construct % that one might 
%hope
 would  equal the join % $a_1 \vee a_2 \vee \cdots \vee a_i$
  (namely the unique least upper bound) 
of % this same collection
  $a_1,\dots ,a_i $. 
%of elements all covering $u$.  % in the event that $Q$ is a lattice.   
%A consequence of our results later will be that t
%The nonrevisiting property for $G(P,\bc )$ would imply that the join of $a_1,\dots ,a_i$ is in the face $F$.  
% this equality of the pseudo-join and the join of $a_1,\dots ,a_i$.
% by virtue of forcing $a_1\vee \cdots a_i$. % equality.
% would  hold 
%if the directed paths in $G(P,\bc )$  could  not revisit any faces. % they have departed from.   

In Theorem ~\ref{pseudo=join}, we prove this equivalence  of the  join and the  pseudo-join operations:   % of a collection of atoms 
%under the  conditions described next:
%  in Theorem \ref{pseudo=join}:
% We refer readers to 
%Section ~\ref{disc-morse-bg-section} for  further background review e.g. regarding posets as well as  polytopes.
 
  \begin{thm}\label{pseudo-is-join}
If $P$ is a simple polytope and $\bc$ is a generic cost vector such that $G(P,\bc)$ is the Hasse diagram of a  lattice $L$, then the pseudo-join 
of any collection of atoms equals the join of this same collection of  atoms.   Moreover, this also holds for each interval $[u,v]$ in $L$. 
\end{thm}

%(perhaps with minor modifications to their  statements).  % to better fit the setting). %or have useful analogues) 
%for  polytopes that need not be simple, 
%However, our  proofs heavily utilize  the hypothesis that our polytopes are simple.  
%The combination of 
%requiring $P$ to be a simple polytope and $G(P,\bc )$ to be a Hasse diagram 
%seems to be quite powerful  when taken together. %, as we hope our results will illustrate.

%\begin{rk}
%Our  upcoming results  hint at the distinct possibility for simple polytopes with $G(P,\bc )$ the Hasse diagram of a lattice  that requiring none of the  directed paths ever to depart and revisit any  low dimensional face of $P$ may force this same face nonrevisiting property 
%% might    force directed path 
%%nonrevisiting
% for  the  higher dimensional faces  as well.  Specifically,
% our result regarding  faces $F\subseteq G$ for $F$ a  codimension one face  in $G$ appearing as Lemma ~\ref{not-same-source-sink}  is
%  designed to facilitate upward propagation  
% in dimension within proofs.  Indeed this does yield  our proof of Conjecture ~\ref{main-conj} 
% for 3-polytopes appearing as Theorem ~\ref{case-of-3-polytopes}.  
% We are hopeful that with more effort and additional  insights Lemma ~\ref{not-same-source-sink} should 
 % have further applications to higher dimensions as well. % going well beyond dimension 3. %than that.
%\end{rk}

%Next we turn to a topological consequence of 
Theorem \ref{pseudo-is-join} combined with the Quillen Fiber Lemma also leads to the following topological property for lattices whose Hasse diagrams arise as 1-skeleta of simple polytopes:
 %A more topological result also appearing in  our paper,  
 %Theorem ~\ref{simple-polytopes}, 
% a result which 
%  in Section ~\ref{polytopes-section},    
%hints  that  perhaps situations where not all pivot rules are efficient could in some cases be  
%detectable using poset topology:

\begin{thm}\label{thm-2}
If $P$ is a simple polytope and $\bc $ is a generic cost vector such that $G(P,\bc )$ is the Hasse 
diagram of a lattice $L$, then   each open interval  $(u,v)$ 
in $L$ has order complex which is homotopy equivalent to a ball or a sphere of some dimension.  Therefore, the M\"obius function $\mu_L(u,v)$ only takes values $0,1,$ and $-1$.
\end{thm}

This is proven  as Theorem \ref{simple-polytopes}.  
Applications are given  %of Theorem ~\ref{thm-2} 
to the
  weak Bruhat order,  the Tamari lattice, and to the $c$-Cambrian lattices in Theorems \ref{weak-theorem}, \ref{Tamari-theorem} and \ref{Cambrian-theorem}, respectively.
%  Section \ref{example-section}.
%\begin{rk}
Posets % with Hasse diagram $G(P,\bc )$  
meeting the hypotheses of  Theorem ~\ref{thm-2}  need  
not  be shellable, as  explained  in Remark ~\ref{almost-never-shellable}, so other methods are indeed needed to understand their topological structure.
%, necessitating  other methods % besides shellability  
%to determine their  topological structure.   %One  such example  that is not shellable is weak order, and one expects many more.  
 % \end{rk}

\begin{rk}\label{which-Hasse}
 %As a different sort of motivation for our work, we note that p
 People sometimes ask whether the Hasse diagrams for  a given family of posets are realizable  as 1-skeleta of polytopes (or more specifically of simple polytopes).
 %, and sometimes they ask more specifically for them to be 1-skeleta of simple polytopes.   
 Well-known examples having  this property  include the weak order, the Tamari lattice, and more recently the Cambrian lattices (arising from the theory of cluster algebras). 
 % There are other  posets conjectured to have this property as well.  
 %;  for instance this was 
 %asked and later answered for the 
 %Tamari lattice.  
Theorem ~\ref{thm-2}   gives a  %might 
%  could be helpful in terms of  providing a  %give a new perspective on potential 
  necessary condition  for a Hasse diagram of a lattice to be the 1-skeleton of a simple polytope.
  %  It  seems plausible that the hypothesis in Theorem ~\ref{thm-2}  that the polytopes under consideration be simple polytopes might not be necessary,  %could be relaxehold for all polytopes whether or not they are simple,
  % in which case one would obtain a necessary condition for a Hasse diagram of a lattice to be the 1-skeleton of an arbitrary polytope. % (as opposed to a simple polytope).
  % particularly if 
 % Theorem ~\ref{thm-2} could be generalized beyond simple polytopes.  
\end{rk}

 %The observations % as well as our other  upcoming 
% and results discussed above %, as well as other more technical related results,   %discussed later    %relating  being a Hasse diagram to 1-dimensional face nonrevisiting
  %led us to  ask  %believe that it could provide useful insight into the question of 
  %what conditions on a polytope and a cost vector would  ensure that  linear programming is efficient  regardless of choice of pivot rule.
    In light of these % results  together with 
    and  other  results  %described above as well as other (somewhat technical)  results 
    proven in this paper,  we   
  make the following conjecture:

  \begin{conj}\label{main-conj}
  Given a simple polytope $P$ and  a generic cost vector $\bc $ such that $G(P, \bc )$ is the Hasse diagram of a lattice,  the  directed paths in $G(P,\bc )$ cannot revisit any face.  % they have left.  
  That is, any directed path $v_1\rightarrow v_2\rightarrow \cdots \rightarrow v_k$ with $v_1$ and $v_k$    in a face $F$ % of $P$ 
 has  $v_i \in F$ for $1\le i \le k$. 
  \end{conj}
  
  %\begin{rk}
%Theorem ~\ref{pseudo-is-join} is  proven % later in the paper 
%as Theorem ~\ref{pseudo=join}. 
% The fact that this result holds   is more subtle than it 
%might appear at first glance.
%This  theorem 
%might perhaps sound like something that should obviously be true; however,  there are subtleties to be overcome.  
 %this would likely be based upon assumptions about nonrevisiting that at a minimum are far from obvious.  Specifically, 
%Santos' result  in \cite{Sa}  that the nonrevisiting path conjecture (see Conjecture ~\ref{nonrevisit-conj}) is false demonstrates how statements  that would be very useful for proving  Theorem ~\ref{pseudo-is-join} that  %in some cases  
 %might sound  like they should  obviously hold  
 %are in fact false.   %Thus, Theorem ~\ref{pseudo-is-join} is more subtle than it might appear at first glance.  
 %   Even the  very special case of Theorem ~\ref{pseudo-is-join} with just two atoms in an arbitrary interval $[u,v]$  seemed rather  challenging to prove.  %  If the directed graph version of the nonrevisiting path conjecture were known to hold, that would make Theorem ~\ref{pseudo-is-join} much easier to prove.  
%\end{rk}
%Our proposed conditions do suffice for 3-dimensional polytopes.
%
%We note that 
Conjecture \ref{main-conj} would directly imply Theorem \ref{pseudo-is-join} above. % would  follow directly from Conjecture \ref{main-conj}.
%;  we 
%believe that Theorem \ref{pseudo-is-join} constitutes a significant step towards proving Conjecture \ref{main-conj}, and particularly so when taken in conjunction with our other results.  
  We  do prove  part of  Conjecture \ref{main-conj} (see Corollary ~\ref{source-face-sink-face}), namely we prove the desired face nonrevisiting property for those faces
  $F$  of $P$ containing  either the source or the sink of $P$. % as a vertex within $F$. %,  as well as  proving  other  related structural results.
  % that may be steps 
  %regarding such lattices $L$ seem to be steps 
  %towards a proof of Conjecture 
  %\ref{main-conj}.  %  or more generally a poset.  We believe these results should be 
 % steps towards a proof of  Conjecture \ref{main-conj}.   %These  upcoming results % towards proving this conjecture
  We use this  together with other structural results in our paper to deduce   Conjecture ~\ref{main-conj}  in  the case of 3-dimensional polytopes, carrying this out  
  %, we note that our results do show this quite readily in the special case of 3-dimensional polytopes.  
  in  Theorem ~\ref{case-of-3-polytopes}.   % Our general  results regarding face nonrevisiting also enable us to
  We also   prove  Conjecture ~\ref{main-conj}  
  for spindles (see Definition ~\ref{spindle-def})  whose two distinguished vertices are the source and sink of the polytope, doing this in Theorem \ref{santos-okay}.  
  %, again utilizing our more technical results to this end.
%A consequence of  Theorem ~\ref{santos-okay}  is   that the known counterexamples to the Hirsch Conjecture cannot be simple polytopes with $G(P,\bc )$ the Hasse diagram of a poset  with the two distinguished vertices of the spindle as source and sink. 
% In  particular, this implies that  % the known examples  of 
% $d$-dimensional spindles with $n$ facets where  % having the property that 
%the  distance between the two distinguished vertices $v_1$ and $v_2$  is greater than $n-d$ %(where $n$ is the number of facets  in the polytope and $d$ is the 
%%%%dimension of the polytope)
%   cannot   meet the hypotheses for Conjecture ~\ref{main-conj} with respect to cost vectors having $v_1$ and $v_2$ as source and sink. %, respectively.
  % be simple polytopes with $G(P,\bc )$ the Hasse diagram of a lattice with the distinguished vertices as source and sink.
    In other words,  we have shown that  %we prove in Theorem ~\ref{santos-okay}  that 
    the established  method that has led to  %none of the known 
    counterexamples to the Hirsch Conjecture (see ~\cite{Sa})  cannot   yield counterexamples to  Conjecture ~\ref{main-conj}.
    Given  any simple polytope $P$ and any generic cost vector $\bc $ such that $G(P,\bc )$ is the Hasse diagram of a lattice, Conjecture \ref{main-conj} would imply that no directed path in $G(P,\bc )$ could have length more than $n-d$; that is, it  would  imply the  Monotone Hirsch Conjecture (cf. Conjecture \ref{monotone-Hirsch}) for such $P$ and  $\bc $.  % Motivated by reading a preprint version of our paper, 
  Dominik Preu\ss \hspace{.02in}  recently informed us  that  %after seeing our paper  he was indeed able to prove
  he has proven  the Monotone Hirsch Conjecture % (see Conjecture \ref{strict-monotone-Hirsch})
    for $P$ a simple polytope with cost vector $\bc $ such that $G(P,\bc )$ is the Hasse diagram of a lattice, motivated to do so by our work (private communication from Dominik Preu\ss).     Preu\ss \hspace{.02in}  has  provided his  proof  %of the monotone Hirsch 
  %Conjecture for such $P$ and $\bc $
   as Appendix A to our paper.  
  % Our conjecture, which is stronger than the monotone Hirsch Conjecture, still remains open at the time of this writing.  

In Section  ~\ref{disc-morse-bg-section}, we review background. 
% including 
%the Hirsch Conjecture, the Nonrevisiting Path Conjecture, and the Strict Monotone Hirsch Conjecture.
%, three famous conjectures  related to the diameter of  polytopes.  
 Section 
~\ref{def-section} introduces and develops  seemingly  new  (or at least not widely known) notions 
 to be used later.      Section ~\ref{polytopes-section} 
gives the proofs of our main technical results.  This includes  Theorems ~\ref{pseudo-is-join}  %~\ref{thm-1} 
and ~\ref{thm-2}  mentioned above as well as a number of
consequences  and  related results.  Section ~\ref{nonrevis-subsection} 
gives applications to two  important classes of polytopes, namely the  3-polytopes and the  spindles.  Section ~\ref{example-section}
%Next we give  
shows how other  well-known families of  polytopes also   fit into our  framework. %, 
%as justified later in 
%Section ~\ref{application-section}, 
%namely %applications to 
%permutahedra, associahedra and 
%generalized associahedra; the posets on their 1-skeleta are the weak order, the Tamari lattice and Cambrian lattices, respectively. % gives applications to well-known families of 
%posets and polytopes.    
Section ~\ref{zonotope-section} turns to the case of zonotopes, where
especially clean results are possible.   
Section ~\ref{posets-from-shellings} extends  our results %  from orientations on 1-skeleta of simple polytopes induced by  generic cost vectors 
to more general acyclic orientations derived  from shellings. % of  dual simplicial polytopes.
%  In the process, we also give  results 
%regarding  an interesting seeming  class of posets derived from shellings.
%We conclude  with  f
Further questions and remarks appear  in Section ~\ref{further-section}.  The paper concludes with the aforementioned  appendix by Dominik Preu\ss.
%, proving the monotone Hirsch conjecture for any simple polytope $P$ and any cost vector $\bc $ such that $G(P,\bc )$ is the Hasse diagram of a lattice.  

%\section{Acknowledgments}

The author is indebted to Karola M\'esz\'aros for several  invaluable conversations at 
early stages of this project.   
She also  thanks 
%the Banff International Research Station for providing a stimulating
%environment in which to begin work that led to this.
%
%The authors  are  grateful to 
Georgia Benkart, Stephanie van Willigenberg,  Monica Vazirani, and  the Banff International Research Station (BIRS) for conducting an inspiring  workshop entitled
Algebraic Combinatorixx  % in May 2011
 for female researchers in algebraic combinatorics.
  %with the goal of helping  establish new and fruitful collaborations.  
In some sense, this project  grew out of discussions that began at that workshop. 
   
 The author also thanks  Louis Billera, Heather Dye, Gil Kalai, 
 Nathan Reading,  Victor Reiner, Francisco Santos, Bridget Tenner, and G\"unter Ziegler 
 for extremely helpful discussions, references, and in some cases counterexamples to questions she raised while working on this project.   She thanks the anonymous referees for lots of highly valuable  feedback that made the paper much better.   
 Finally, the author wishes to express her gratitude for the excellent mentoring of  Victor Klee  and the combinatorics group at the University of Washington during her term from 1999 to 2001 in a postdoctoral position there and for introducing her  to the interesting topic  of the Hirsch Conjecture back then.

\section{Background} 
\label{disc-morse-bg-section}

  A {\bf cover relation}  $u\prec v $ % occurs 
   in a finite partially ordered set (poset)  $Q$ is  $u\le v$ in $Q$  with the  requirement that $u\le z \le v$ implies either $z=u$ or $z=v$.   The {\bf Hasse diagram} of a finite poset $Q$   is the directed graph with directed edges $u\rightarrow v$  if and only if $u \prec v$ in $Q$.    If a poset has a unique minimal element, denote this element by $\hat{0}$.  If  a poset has a unique maximal element, denote this by $\hat{1}$.   An {\bf atom} in a poset $Q$ with $\hat{0}$  is any $a\in Q$ satisfying  $\hat{0} \prec a$.  Likewise  a {\bf coatom} in a poset $Q$ with $\hat{1}$ is any  element $c$ satisfying  $c\prec \hat{1} $. % is called a {\bf coatom}.  
 %  The {\bf dual} of a poset $Q$, denoted $Q^*$, is a partial order having the same set of elements as $Q$ with order relation $u\le v$ in $Q^*$ if and only if $v\le u$ in $Q$.
  
  If $x,y \in Q$ have a unique least upper bound, this is called the {\bf join} of $x$ and $y$, denoted $x\vee y$.  If $x,y \in Q$ have a unique
  greatest lower bound, this is the {\bf meet} of $x$ and $y$, denoted $x\wedge y$.
  A poset $Q$  is a {\bf lattice} if each  pair of elements $x,y\in L$  have a meet and a join.
  %there is a unique least upper bound for $x$ and $y$ which is itself  contained in $L$, denoted $x\vee y$, and there is also a unique greatest lower bound for $x$ and $y$  that is contained in $L$, denoted $x\wedge y$.  In this case, $x\vee y$ is  called the {\bf join} of $x$ and $y$,  and $x\wedge y$ is called the {\bf meet} of $x$ and $y$.  
  Any partially ordered set $Q$ gives rise to a {\bf dual} partial order on the same set of elements, denoted $Q^*$,  with $u\le v$ in $Q^*$ if and only if $v\le u $ in $Q$.

\begin{rk}
  A key example throughout  this paper of poset duality will be as follows: whenever $G(P, \bc )$ is the Hasse diagram of a poset $L$, then the dual poset  $L^*$ will have as its Hasse diagram the directed graph $ G(P, -\bc)$.
\end{rk}
  
 Denote by $(u,v)$ the subposet of $Q$ comprised of those $z\in Q$ satisfying $u < z < v$.  This is known as the {\bf open interval} from $u$ to $v$.  Likewise, we define the {\bf closed interval } from $u$ to $v$, denoted $[u,v]$, to be the suposet of elements $z\in Q$ satisfying $u\le z\le v$.  Define the 
 M\"obius function of $Q$, denoted $\mu_Q$,  recursively by setting $\mu_Q(u,u) = 1$ for each $u\in Q$ and 
 $$\mu_Q(u,v) = - \sum_{u\le z < v} \mu_Q(u,z).$$
 
The {\bf order complex} of a finite poset $Q$, denoted $\Delta (Q)$,   is the simplicial complex
whose $i$-faces
are  the chains 
$ v_0 < \cdots < v_i  $
of $i+1$ comparable
elements of $Q$.  We let $\Delta (u,v)$ (or $\Delta_Q (u,v)$) denote the order complex of the open interval $(u,v)$ in $Q$.  By definition, a  poset and its dual poset have the same order complex.  It is well-known (by Hall's Theorem) that
$\mu_Q(u,v) = \tilde{\chi }(\Delta (u,v)) $ 
where $\tilde {\chi } $ is the {\bf reduced Euler characteristic} of 
$\Delta (u,v)$, namely $$\tilde{\chi }(\Delta (u,v) ) = -1 + f_0(\Delta (u,v)) - f_1 (\Delta (u,v )) + f_2 (\Delta (u,v))  - \cdots $$ for $f_i(\Delta )$ the number of $i$-dimensional faces in $\Delta $.
Sometimes we will speak of the homotopy type of a poset or of a poset interval, by which we  mean the homotopy type of the order complex of that poset or that poset interval.  
See e.g. \cite{st} %\cite{ec1} 
for further background on posets.  

%Next we review notions % we will need related to polytopes.  
A {\bf polytope} is any subset of $\reals^d$  arising as the 
 convex hull of a finite set of vertices in $\reals^d$ for some $d$; equivalently, a polytope
 is any bounded set given by a system of non-strict linear inequalities, or in other words any bounded
 set expressible as $\{ {\bf x} \in \reals^d | A {\bf x} \le {\bf b} \} $ for some choice of   constant 
  $n\times d  $ real matrix $A $ and  some choice of 
 constant vector ${\bf b}\in \reals^n$.
 We call a polytope a {\bf d-polytope} if there is a $d$-dimensional affine space containing the polytope but there is  not a  $(d-1)$-dimensional affine space containing this same polytope.      %A good source for background on polytopes is
 % Before giving more background on polytopes, we note that all of these notions related to polytopes as well as linear programming
  %are discussed in more depth in 
  %\cite{zi}.
 An excellent reference to learn about polytopes in more detail than we will give here is \cite{zi}.
 
 Any hyperplane $H$  that intersects a polytope $P$  nontrivially but has all points of $P$  either contained  in $H$  or on one side of $H$ is called a {\bf bounding hyperplane} of $P$.  
 The intersection of a bounding hyperplane with a polytope is called a {\bf face} of the polytope.  
 A  maximal face in the boundary of a polytope is called a {\bf facet}.

A  polytope is {\bf simplicial}  if each face in its boundary is a simplex.  
A polytope $P$  is  {\bf simple}
%Another formulation of this notion of simple polytope, 
%one which we will find convenient to use frequently in this paper, is that  a polytope
if for each vertex $v\in P$ and each collection of $i$ edges emanating outward from $v$, there is an 
 $i$-dimensional face of $P$ containing $v$ and all these edges incident to $v$.
 %; this is equivalent to requiring the dual polytope to $P$ to be a simplicial polytope.
 %while this is not the traditional way simple polytopes are typically defined, it is equivalent and will be a very useful working definition for proofs in this paper. 
Equivalently,  a  polytope $P$ is simple if 
its dual polytope (briefly discussed next and defined more precisely e.g. in \cite{zi})  %(as defined e.g. in \cite{zi}) 
is a simplicial polytope.

The {\bf face poset}, denoted $F(P)$,   of a polytope $P$  is the partial
order on faces  with $\sigma < \tau $ if and only if  $\sigma $ is in the boundary of 
$ \tau $.   %We make the convention throughout this paper of excluding the empty face.   
% For $K$ a polyhedral  complex,  the order complex of the face poset of $K$  is the first barycentric subdivision of $K$, and in particular is homeomorphic to $K$.   
Each polytope
$P$  has a {\bf dual polytope}, denoted $P^*$ with face poset satisfying  $F(P^*) = (F(P))^*$. 

\begin{rk}\label{caution-remark}
 As a word of caution, note that in cases where  $G(P,\bc )$ is the  Hasse diagram of a poset $Q$, the face poset $F(P)$ of the polytope $P$  is typically  a completely different  poset than  $Q$.  In this case negating the cost vector $\bc $ while preserving the polytope $P$  yields a   directed graph  $G(P,-\bc)$ on the 1-skeleton 
 of $P$  which    is the Hasse diagram of  $Q^*$. 
 \end{rk}
%the dual of a polytope $P$ is typically not the same polytope as $P$, 
%whereas dualizing  the poset with Hasse diagram $G(P,\bc )$ when $G(P,\bc )$  is a Hasse diagram  does 
%preserve the polytope with this 1-skeleton, with this poset duality simply reversing  the arrows  on the edges of $P$.
 %%%obtained by taking the convex hull of the barycenters of the facets of $P$ (and rescaling as needed so that $P^{**} = P$).  
%%% Notice that $F(P^*)$ has the same elements as $F(P)$ with $u\le v$ in $P$ if and only if $v\le u$ in $F(P^*)$.

 \begin{defn}\label{spindle-def}
  A {\bf spindle} is a polytope  $P$ with a distinguished pair of vertices $u$ and $v$ such that each facet  of $P$  includes either $u$ or $v$.  The dual polytope to a spindle is called a {\bf prismatoid}, and it is characterized by the property that it has two distinguished facets such that  every vertex belongs to one or the other of these two facets.
 \end{defn}

 A {\bf zonotope} is a polytope arising as a linear projection of a cube of some dimension.  In other words,  
 a zonotope is a Minkowski sum of line segments.   %See e.g. \cite{zi} for  further background on polytopes.

%We refer readers to \cite{zi} as an excellent for more details and further background in this area.  

%Now let us discuss motivations for our work, first reviewing background this will require.  
%The aim of  linear programming is  to maximize (or to minimize)  $\bc \cdot {\bf x} $ for a fixed cost vector $\bf{c}$ over all  possible choices of ${\bf x} \in P$ for $P$ a polyhedron, i.e.,  for $P$ a subset of $\reals^d$ given by a system of (weak) linear inequalities. %  When this subset of $\reals^d $ is bounded, then  $P$ is a polytope.  
%One particularly famous example of a linear programming problem is the traveling salesman problem (see e.g. \cite{Co}).  
 %When $P$ is a polytope, 
 %the {\bf  simplex method} for  linear programming  finds  the vertex of $P$ where this maximum (resp. minimum)  is achieved   by greedily following directed edges of $G(P,\bc )$ until reaching the unique sink (resp. source)  of the directed graph $G(P, \bc )$.  
 In linear programming, a  {\bf  pivot rule} for the simplex method   is a rule for  choosing for each   vertex $v$ of $G(P,\bc )$ which outward (resp. inward)  oriented edge from $v$ to traverse in choosing a  
directed path  to the sink (resp. source).   
%We will henceforth  focus on maximizing $\bf{c}\cdot \bf{v}$ for $\bf{v} \in P$,  since the minimization problem  is completely equivalent.  % by reversing arrows in $G(P, \bf{c})$.  
% There is a large body of literature regarding pivot rules, since the choice of pivot rule for a particular %problem  may greatly impact the effectiveness of the simplex method.  

%From the viewpoint of linear programming, it seems natural and useful  to  ask for conditions on $P$ and $\bc $  that would ensure  that no directed path can revisit any face it has left.  After all, having this nonrevisiting property  on a $d$-polytope $P$ with $n$ facets  would 
%guarantee  % for $P$ a $d$-polytope with $n$ facets 
%that all  possible pivot rules would be efficient 
%in the sense that every directed path in the graph  $G(P,\bc )$ 
%would  reach the vertex $\bf{v} $ where $\bc \cdot \bf{v} $ is maximized in at most $n-d$ steps. 

Now we recall the Hirsch Conjecture,  the Nonrevisiting Path 
Conjecture, % the Strict Monotone Hirsch Conjecture, 
and the Monotone Hirsch Conjecture.  We refer readers e.g.  to \cite{zi} for a more
in-depth discussion of all of these conjectures.

\begin{conj}[Hirsch Conjecture]\label{Hirsch-conj}
For $n>d\ge 2$, let $\Delta (d,n)$ denote the largest possible diameter of the graph of a $d$-polytope with $n$ facets.  Then $\Delta (d,n) \le n-d$.
\end{conj}

\begin{conj}[Nonrevisiting Path Conjecture]\label{nonrevisit-conj}
For any two vertices $u,v$ of a $d$-dimensional  polytope, there is a path from $u$ to $v$ which does not revisit any facet it has left before.
\end{conj}

The  Nonrevisiting Path  Conjecture, proposed by Klee and Wolfe,
  implies the Hirsch Conjecture.
To see this implication, notice  that any directed path from $u$ to $v$ of the type  given by the Nonrevisiting Path Conjecture would involve at most $n-d$ edges, since each edge would depart a facet, with no facet departed more than once, and since the ending vertex  $v$ for the path would still  belong to $d$ facets; thus, each pair of vertices $u,v$ would have a path of length at most $n-d$ between them.  
Counterexamples to the Hirsch Conjecture (and thereby also  to  the Nonrevisiting Path Conjecture)  were first obtained by Francisco Santos  in \cite{Sa}:

\begin{thm}[Santos]
The Hirsch Conjecture is false.  Therefore, the Nonrevisiting Path Conjecture is also false.
\end{thm}

One may still ask 
for sufficient conditions on a polytope for 
these conjectures  to  hold  and for our even stronger  face nonrevisiting property to hold.   
% We believe the requirement that $G(P,\bc )$ be a  Hasse diagram  would be a useful possibility as one such condition to consider, in conjunction with other properties such as   $G(P,\bc )$ being the  Hasse diagram of  a lattice.   
%Some evidence is provided by our upcoming results for one possible answer to this latter question, namely 
% to suggest  that requiring  $G(P,\bc )$  to be the Hasse diagram of a lattice might
%very well suffice for a  simple polytope $P$ 
%to guarantee our  face  nonrevisiting  property for $G(P,\bc )$, namely evidence is provided 
% for our  Conjecture ~\ref{main-conj}.  
%is 
%provided  by our upcoming results.    
%There is also the following related question:
%
%\begin{qn}\label{non-revisit-rk}
%There is  also  the related question of w
%What are sufficient conditions on a polytope $P$ and cost vector  $\bc $ so that  directed paths in $G(P,\bf{c}) $ may never revisit any facet they have left?  
%Notice that this is equivalent to asking for conditions under which  directed paths may  never revisit any facet they have left, by virtue of each face being an intersection of facets.  
%\end{qn}
%
%This question from Remark 
%Question ~\ref{non-revisit-rk}
This latter question is   related  to the % Strict 
Monotone Hirsch Conjecture for polytopes:

%\begin{conj}[Strict Monotone Hirsch Conjecture]\label{strict-monotone-Hirsch}
%Let $P$ be a $d$-dimensional polytope with $n$ facets, and let $\bc $ be a generic linear functional.  Then there is a directed path in $G(P,\bc )$ from the source of $P$ to the sink of %$P$ of length at most $n-d$.
%\end{conj}

%Let us also recall the monotone Hirsch Conjecture, which was proven false by Todd in \cite{Todd}.  
%The counterexample of Todd fails to satisfy the Hasse diagram property, giving one more piece of evidence that failure of this Hasse diagram property may play a role in enabling many of the known pathological examples to exist.   
%helpful for identifying and proving that particular classes of polytopes cannot have the pathological  behavior found in known counterexamples to the Hirsch Conjecture and its close relatives.

\begin{conj}[Monotone Hirsch Conjecture]\label{monotone-Hirsch}
%Let $P$ be any $d$-polyhedron in $\reals^d$ with at most $n$ facets.
Let %$H_u(d,n)$
$H(d,n)$ 
be the smallest integer  $N$ such that for every  $d$-polytope % $d$-polyhedron 
$P$ in $\reals^d$, every  cost vector $\bc \in \reals^d$ in general position with respect to $P$, and 
every  vertex $v\in P$,  there exists a strict monotone path from $v$ to the sink of $P$ using at most $N$ steps.  
%Let $H(d,n)$ be this same quantity in the case that each $P$ is required to be  a polytope.  
Then % $H_u(d,n) \le n-d$ and  
$H(d,n)\le n-d$.
\end{conj}

To see the connection, notice that  any 
%any set of conditions  on a 
polytope $P$ and cost vector $\bc $ such that % that would 
%guarantee that 
no directed path in $G(P,\bc )$ revisits any facet % after leaving it 
must have  $G(P,\bc )$ of  diameter at most $n-d$, by the same reasoning  %(recalled just after Conjecture ~\ref{nonrevisit-conj} above)
  which showed  
that the Nonrevisiting Path Conjecture  implies the % Strict 
Monotone Hirsch Conjecture for polytopes (see the discussion just after Conjecture ~\ref{nonrevisit-conj} above).
%  Thus, any such conditions would yield  the Strict Monotone Hirsch Conjecture for the class of polytopes and cost vectors meeting the hypotheses.

Todd gave the first  counterexample  to the Monotone Hirsch Conjecture in \cite{Todd}; this came from  a 4-polytope with 8 facets and a vertex $v$ requiring at least  5 steps for all  monotone paths from $v$ to the sink.  However, this polytope has an edge from $v$ to the source and then a directed edge from source to sink.  From the standpoint of our own conjecture, it is perhaps worth noting that the existence of this directed  edge prevents  the directed graph in this case from being a Hasse diagram.   
% This gives yet one more example of the known counterexamples to various conjectures in this area failing the Hasse diagram property,  suggesting  the 
%possibility that %question of whether most  if not 
%all polytopes with the Hasse diagram property could in fact satisfy all  of  these conjectures.  
%One of our main goals is to initiate a study of Question ~\ref{non-revisit-rk} from a poset 
%theoretic perspective and to make progress towards resolving this question.   
%We suggest hypotheses on $P$ and $\bc $  that might perhaps suffice to ensure that  no directed path in $G(P,\bc )$ may revisit a face after leaving it; we also prove that our proposed
%hypotheses do at least yield a corollary that such face nonrevisiting would have.   

%  When $P$ is a simple polytope, the Hasse diagram property alone, namely 1-dimensional face nonrevisiting, might actually be enough to get this propagation of nonrevisiting  from low dimensional face nonrevisiting  to high dimensional face nonrevisiting started.  

We conclude this section with some further background on  poset topology.
A map $f:P\rightarrow Q$ from  a poset $P$ to  a poset $Q$ is a {\bf poset map} if $u\le v$ in $P$ implies $f(u)\le f(v)$ in $Q$.  

\begin{thm}[Quillen Fiber Lemma, \cite{Qu}]\label{Quillen-Lemma}
Let $f:P\rightarrow Q$ be  a poset map such that for each $q\in Q$ the order complex  $\Delta (f_{\ge q}^{-1})$ for
%of 
$f_{\ge q}^{-1} = \{ p\in P| f(p)\ge q \} $ is contractible.  Then $\Delta (P) \simeq \Delta (Q) $.
\end{thm}

%Theorem ~\ref{Quillen-Lemma}  was proven in  %See
%\cite{Qu}. % for the proof of the result about poset maps given next. 
% Recall that a
 A  {\bf dual closure map} is a poset map $f: P \rightarrow P$ with $f(u) \le u$  and  $f^2 (u ) = f(u)$ for all $u\in P$.  Notice that any such $f$ meets the contractibility requirement of the Quillen Fiber Lemma,  by virtue of  each $u\in im(f)$ being a  cone point in  $\Delta (f_{\ge  u}^{-1})$. Thus, $\Delta (im(f)) \simeq \Delta (P)$ in this case.  
 
\begin{rk}\label{our-dual-closure-map}
The poset  map $f$ sending each element $u$ in a finite  lattice to the join of those atoms $a$ satisfying $a\le u$ is a dual closure map with  the further   property that $f^{-1}(\hat{0}) = \{ \hat{0} \} $.  
Thus, the Quillen Fiber Lemma yields $\Delta (P\setminus \{ \hat{0} \} ) \simeq \Delta (im(f) \setminus \{  \hat{0} \} )$ in this case.
\end{rk}

\section{The nonrevisiting property, the Hasse diagram property, pseudo-joins and pseudo-meets}\label{def-section}

In this section, we introduce or in some cases make more precise a few seemingly new 
notions 
which will  play key roles  %be quite useful 
 throughout this paper.   

%We will mainly use these notions in the context of simple polytopes, in which case the smallest face containing a set of $r$ atoms will be exactly $r$-dimensional.

\begin{defn}\label{non-revisiting-def}
A directed graph $G(P,\bc)$ on the 1-skeleton of a polytope $P$ satisfies the {\bf nonrevisiting property} if for each facet $F$ and each directed path $p_F$  that starts and ends at vertices  
in $F$, the path $p_F$ must stay entirely within $F$.   % Remark ~\ref{facets-implies-faces} below gives an equivalent reformulation with facets replaced by faces.
We say that $G(P,\bc )$ satisfies the
nonrevisiting property for $i$-dimensional faces if for each $i$-face $F$ and each directed path 
$p_F$ that starts and ends at vertices of $F$, every vertex of $p_F$ is in $F$.  
\end{defn}

\begin{rk}\label{facets-implies-faces}
This nonrevisiting property for facets  is equivalent to % immediately implies
 the nonrevisiting property for
all  faces of every  dimension  since each  face is an intersection of facets which means that  to depart a face and revisit that face 
would require departing at least one facet containing the face  and then revisiting this same facet.   
%Conversely the nonrevisiting property for all faces implies it for facets since every facet is a face.
\end{rk}

%\begin{notation}
%For $S$ a subset of a poset,
%%%Just as $psj(S)$ denotes the pseudo-join of the collection of elements of $S$, 
%denote by $J(S)$ the join of the  %collection of 
%elements of $S$.
%\end{notation}

%\begin{defn}
%A directed graph $G(P, \bc )$ on the 1-skeleton of a polytope has the {\bf Hasse diagram
%property} if it is the Hasse diagram of a poset.
%\end{defn}

% We next observe  that the non-revisiting property includes the  Hasse diagram 
%property as a special case.

The nonrevisiting property for 1-faces is equivalent to the Hasse diagram property:

\begin{lem}\label{1-d-revisiting}
The nonrevisiting property holds  for 1-dimensional faces of a polytope $P$ with respect to cost vector $\bc $ if and only if 
%, namely
 %the requirement that a directed path cannot depart from an edge 
 %%%1-dimensional face 
% and then revisit  that same edge  later,
%  implies 
%that  the directed graph 
$G(P,\bc)$  is  the  Hasse diagram  of a poset.  
% This poset structure on $G(P,\bc )$  has 
%Specifically, 
%$G(P,\bc )$ is then  the 
%Hasse diagram  for the poset having 
%$u\le v$ whenever there is a directed path from $u$ to $v$ in $G(P,\bc )$.  
%Conversely, if $G(P,\bc )$ is a Hasse diagram then  the nonrevisiting property for 1-dimensional faces of $P$ holds  with respect to cost vector $\bc $.  
\end{lem}

\begin{proof} 
% The point is that 
First we show how the nonrevisiting property for 1-faces implies $G(P,\bc )$ is a Hasse diagram.
Acyclicity of $G(P,\bc )$ ensures  that the directed paths indeed specify comparabilities in a partially 
ordered set.  %a directed path cannot revisit a vertex it has left.  
The  nonrevisiting property for 1-dimensional faces  guarantees  that  a directed path cannot visit the sink of a directed  edge after departing from  the source of that directed edge in a manner other than traversing that edge. %   that departs from the edge itself.  
 This shows that each directed edge gives rise to a cover relation in the poset rather than simply an order relation.

The other direction follows directly from  the  definition of  a cover relation in a poset.
\end{proof}

\begin{ex}\label{klee-minty}
An especially instructive and significant family of 
simple polytopes failing the Hasse diagram property are the 
Klee-Minty cubes (cf. \cite{KM}). 
% These were the first known polytopes exhibiting  that the simplex method from optimization is not always efficient.   
% See \cite{KM}, or e.g. see \cite{GHZ}  where (among other things)  a particularly helpful  illustration of one of
 %these polytopes appears.  
%
%The Klee-Minty cubes  are  hypercubes with carefully chosen (non-standard)  locations for the vertices.  They were  designed so that there exists a cost vector $\bc $ such that   $G(P,\bc )$  has a directed path from source to sink that visits all of the  vertices.  
%Thus, these are $d$-polytopes with a  directed path in $G(P,\bc )$ visiting $2^d$ vertices, exhibiting that a pivot rule exists which makes the simplex method highly non-efficient.  
%
Klee-Minty cubes are polytopes  $P$  that are realizations of $d$-dimensional hypercubes admitting  a  cost vector $\bc $ 
%such that  $P$ is a realization of a $d$-dimensional cube 
with  the property  that a  directed path exists  in  $G(P,\bc )$ that  visits all $2^d$ vertices of $P$.    These were historically  the first examples of polytopes  demonstrating that 
the simplex method for linear programming is not a polynomial time algorithm. %; they were designed to demonstrate a shortcoming of  the simplex method.

The $d$-dimensional Klee-Minty cube may be realized for any fixed $\epsilon $ satisfying $0 < \epsilon < 1/2 $ as $\{ (x_1,\dots ,x_d)\in \reals^d | 0\le x_1\le 1; \epsilon x_{i-1} \le x_i \le 1-\epsilon x_{i-1}\hspace{.04in}  {\rm for } 
\hspace{.04in}  i>1 \} $.      For example, 
the  2-dimensional Klee-Minty cube in $\reals^2$ with $\epsilon = 1/3$  is the convex hull of vertices $(0,0), (0,1), (1, 1/3), (1,2/3)$; for cost 
vector $\bc = (0.1,1)$,  the resulting graph $G(p,\bc )$ has a directed path $p$ from $(0,0)$ to $(1,1/3)$ to $(1,2/3)$ to  $(0,1)$  that visits all 4 vertices.  Notice that this 2-polytope also has  a directed edge $e$ from $(0,0)$ to $(0,1)$, which means that the aforementioned  
directed path $p$ departs from the edge $e$ at $(0,0)$  and later revisits $e$ at $(0,1)$.  This demonstrates that  $G(P,\bc )$ is not  a Hasse diagram in this case.
 %The failure of the Hasse diagram property even within 2-dimensional faces in Klee-Minty cubes seems to be at the heart of what allows there to be such long directed paths in these polytopes.  
%Careful examination of this important family of examples  led us to to regard  these examples  as  quite suggestive
%We believe  that the Hasse diagram requirement on $G(P,\bc )$  may  preclude long directed paths of just the type that lead to inefficiencies in the simplex method.
% while being a reasonably checkable condition and one  that  many important families of polytopes may satisfy.  
%One might hope to glean from our
% results  some new understanding of  conditions that preclude the sorts of pathologies that can make the (undirected or directed) diameter of a polytope surprisingly large.   
 %See for instance  \cite{GHZ} 

 For further  background and properties of 
Klee-Minty cubes, including  a  helpful illustration of a 3-dimensional Klee-Minty cube, we refer readers to
\cite{GHZ}.
\end{ex}

Recall for a polytope $P$ and generic 
cost vector $\bc $  that the {\bf source} of a face $F$ is the vertex  
$v \in F$ minimizing $\bc \cdot v $ while the {\bf sink} of $F$ is the vertex $w\in F$ maximizing
$\bc \cdot w$.  The source of $F$ is also the unique vertex  % of $G(P,\bc )$
 in $F$ only having outward oriented edges to other vertices of $F$ while the sink is
the unique vertex 
%of $G(P,\bc )$
 in  $F$ only having inward oriented edges to it from other vertices of $F$.  The % implicit
uniqueness assumption that is  implicit  in these notions above  is justified as follows:

\begin{rk}\label{face-unique-source-sink}
It is well known  (cf.  Theorem 3.7 in  \cite{zi}) and straight-forward to see that there is a unique source and  a unique sink in the directed graph   on the 1-skeleton of any face $F$  of a   polytope $P$ obtained by restricting $G(P,\bc )$ for a generic $\bc $  to the face $F$. % has a unique source and unique sink.  
\end{rk}

Next is a simple  observation that is surprisingly useful in various proofs later in the paper.

\begin{rk}\label{2-face-source-sink}
%  Remark ~\ref{face-unique-source-sink} 
   %The previous remark 
% The uniqueness of source and sink  for each 
% face  implies in particular   that t
 The directed graph $G(P,\bf{c})$  restricted to any 2-dimensional face  $F$ 
consists of   two directed paths from the unique source of $F$ to the unique sink of $F$.    These paths are disjoint except at their  two endpoints, i.e., the source and sink.  
% We specifically point out  this 
%property for  2-faces because we  will use it repeatedly later. 
 \end{rk}

Next we  introduce and justify existence of 
the following pair of notions which are dual to each other  (the first of which was already introduced more informally in the introduction):

\begin{defn}\label{pseudo-def}
Consider any simple
polytope $P$ and any generic  cost vector $\bc$ such that $G(P,\bc )$ is the  Hasse diagram of a poset $Q$.
 %which induces a partial order $L$  on the 1-skeleton of $P$ by letting $x\prec y$ for each edge $e_{x,y}$ in $P$ with $\bc(x) < \bc(y)$. 
  Define the {\bf pseudo-join} of any collection $S$ of  atoms of an interval $[u,v]$ in $Q$, denoted $psj(S)$,   to be the sink of the 
unique smallest  face $F_S$ of $P$  that contains $u $ and all of the elements of $S$.   
%Denote this sink  by  $psj(S)$. 
 Define the {\bf pseudo-meet} of any collection  $T$ 
of coatoms in  an interval $[u,v]$ in $Q$, denoted $psm(T)$,   to be the source vertex  of the smallest face  $G_T$
containing $v$ and all of the elements of $T$.
%Denote this source  by  $psm(T)$.
 \end{defn}

Next let us justify the existence (and uniqueness) of such pseudo-joins.

\begin{lem}
Consider  a simple polytope $P$ and a generic cost vector $\bc $ such that $G(P,\bc )$ is the Hasse diagram of a poset $Q$.  Any collection $S$ of atoms in 
any interval $[u,v]$ in $Q$ has a unique pseudo-join.
\end{lem}

\begin{proof}
%The existence  and uniqueness of   pseudo-joins (and dually of  pseudo-meets)  is justified by the following facts.  Since $P$ is simple, there does exist an $i$-face containing  First we note that  there is a unique smallest face containing a collection of vertices  the intersection of any two faces containing these vertices also contains all these vertices.  On the other hand, each face has a unique sink and a unique source.    
%%%%\begin{enumerate}
%%%\item
For any vertex $v$ in $G(P,\bc )$  
and any collection $S$ of neighboring vertices that cover $v$ in $Q$, consider the  edges $e_1,\dots ,e_i$ emanating outward from $v$ whose other 
endpoints are the elements of $S$.   A face of $P$ will contain all of  the vertices in $S\cup \{ v\} $ if and only if it contains all of the edges $e_1,\dots ,e_i$. 
By the definition of  simple polytope, % ensures that %  then by definition of simple
there exists an   $i$-face $F$ containing $v$  and all of these edges $e_1,\dots ,e_i$.    This is necessarily the unique smallest face containing $v$ and all  of the edges
$e_1,\dots ,e_i$,
due to  any such face  needing to be  %necessarily  being 
 at least $i$-dimensional combined with  % and also using
  the fact that the intersection of any two faces containing all of these edges  $e_1,\dots ,e_i$ 
will also contain all of  these same edges.  
%This must be the unique  smallest face containing all these edges since a
%%%%Since any face of $P$  containing all these edges must be at least $i$-dimensional,  $F$  must be the  unique smallest face containing all these edges. 
%%%%since otherwise the intersection of $F$ and another  minimal face  containing all these edges would be a  face that is less than $i$-dimensional  also 
%%%strictly  lower dimensional  
%%%%face 
%%%%containing all these edges.   %Since $G(P,\bc )$ is a Hasse diagram, % By the Hasse diagram property, 
%By construct,
% $F$  must then be  the unique smallest face  of $P$ containing $v$  and  all the elements of $S$.  
Since $F$ has a unique sink, 
this  makes the sink of $F$ 
the desired pseudo-join.
%%%%Now that we have the existence and uniqueness of this 
%%%%face $F_S$, i
%%%%The sink of $F_S$ is the desired pseudo-join and by uniqu is unique.  %, 
%%%due to  having the minimal possible number of edges emanating out from $v$  that are all contained in an $i$-face.
%%%%\item  
%%%%Whenever  faces $F$ and $G$  both contain a collection of atoms, then $F\cap G$  will be another face  also containing all these atoms. 
%%%%\end{enumerate}
\end{proof}

 \section{Poset theoretic results regarding 
1-skeleta of simple  polytopes}\label{polytopes-section}

In this section, we  develop  a series of 
general results about directed paths in the graphs $G(P,\bc )$ under the assumptions  that $P$ is a simple polytope, $\bc $ is a generic cost vector, and $G(P,\bc )$ is the Hasse diagram of a poset $Q$.  
% 1-skeleta  of  simple
%polytopes  that are also Hasse diagrams of posets.  
This leads to results later  in this section 
on the topological  structure of the  order complexes of such posets.

\begin{lem}\label{not-same-source-sink}
Let $P$ be   a simple polytope  with faces $F\subseteq G$ satisfying $dim (G) = dim (F) +1$.  Let $\bc $ be  a generic 
cost vector  such that
$G(P,\bc )$  % the directed graph that $\bc$ induces on the 1-skeleton of $P$
  is the Hasse diagram of a poset $Q$.    %Then   there cannot be vertices 
 Given vertices  $v,w\in F$ with  a directed path $p_F$ from $v$ to $w$  fully contained  in $F$, $G(P,\bc )$ cannot have 
 an edge from $v$  directed outward to some vertex $v'\in G\setminus F$  and  an edge directed  from some 
 $w'\in G\setminus F$   to $w$.
\end{lem}

\begin{proof}
Suppose there is a directed path $p_F$ of the type we aim to exclude.  
Since $P$ is simple, each vertex  $u$  in $p_F$  has exactly one edge $e_u$  
incident to it whose other endpoint is  in $G \setminus F$.   If $e_u$  is oriented outward from $u$, denote this by  $o(u) = +1$, whereas we  say $o(u) =  -1 $ when $e_u$ is oriented towards $u$.  
Our hypotheses give us   that  $v\in p_F$  has $o(v) = +1 $
 while  $w \in p_F$  has $o(w) = -1 $.    This forces 
%  allows us to  apply  the discrete version of the  intermediate value theorem to  deduce 
   the existence of  two consecutive vertices $v_1\rightarrow v_2$  in  the directed path $p_F$ with  $o(v_1) = +1$ and  $o(v_2) = -1$. 
   % That is,  
   %$v_1$ has an  edge $e_{v_1}$  directed outward from it  to some $x_1\in G\setminus F$ while 
   %$v_2$ has an  edge $e_{v_2}$  directed inward to it from some vertex $x_2\in G\setminus F$.  
   
   %2023: just changed $P$ to $F$ various places in next line
 We  now  show   that the  edges $e_{v_1, x_1}$ and $e_{v_2,x_2}$ from $v_1\in F$ to some $x_1\in G\setminus F$ and from some $x_2\in G\setminus F$ to $v_2 \in F$ 
 must both be contained in  a single  2-dimensional face  $F(e_{v_1},e_{v_2})$ 
 in $G$ that also contains the edge $e_{v_1,v_2}$. % from $v_1$ to $v_2$. 
 %and that  the Hasse diagram  property  therefore must fail on the directed 1-skeleton restricted to $F(e_{v_1},e_{v_2})$.  
 The fact that $P$ is simple implies 
 % for any vertex $v$  and any pair of edges both incident to $v$, there is a unique 2-dimensional face containing both edges.  This implies 
 that  the pair of edges $e_{v_1,v_2}$ and $e_{v_1,x_1}$  are both contained in  a 2-face $F(v_1,v_2,x_1)$.  Moreover, $F(v_1,v_2,x_1)\not \subseteq F$ since $x_1\not\in F$.  Likewise there exists  a
 % ikewise obtain from  the pair of edges $e_{v_1,v_2}$ and $e_{v_2,x_2}$  a 
 2-face $F(v_1,v_2,x_2)$ containing $e_{v_1,v_2}$ and $e_{x_2,v_2}$ with 
 $F(v_1,v_2,x_2)\not \subseteq F$ because 
 $x_2\not\in F$.  
  But the fact that $P$ is simple implies that each edge  $e$ in  $F$ is contained in a  unique 2-face $\sigma $  in $G$  such that 
 % this 2-face is not  contained  in 
  $\sigma \not\subseteq F$; this  follows from 
  %by virtue of  
  each upper interval in the face poset of a simple polytope  being a Boolean lattice and in particular the interval $[e,G]$ being a Boolean lattice where all but one of its atoms is in $[e.F]$.  
  Applying this observation   to the edge $e_{v_1,v_2}$ yields
  that $F(v_1,v_2,x_1)$ and $F(v_1,v_2,x_2)$ must both be this  unique 2-face containing $e_{v_1,v_2}$ and  not contained in $F$.  Thus, %is shows that 
  %there exists  a unique  2-face 
  $F(v_1,v_2,x_1) = F(v_1,v_2,x_2)$ with these both  being  the unique 2-face   which contains all three   
  edges $e_{v_1,v_2}, e_{v_1,x_1}, $ and  $e_{v_2,x_2}$.   %But a
  
 % Applying   Remark ~\ref{face-unique-source-sink}  %and  ~\ref{2-face-source-sink} 
  %  to  this face $F(v_1,v_2,x_1)$, we deduce 
 Observe then  that $v_1$ must be the unique source for $F(v_1,v_2,x_1)$ and that  $v_2$ must be the unique sink for $F(v_1,v_2,x_1)$.  But there is a directed edge from $v_1$ to $v_2$ in 
  $G(P,\bc )$, namely  $e_{v_1,v_2}$. %in the path $p_F$.  
     %and that  the directed edge 
     % Since this edge is from the source to the sink of $F(v_1,v_2,x_1)$, t
      This   edge  must constitute  one of the two directed paths from the source to the sink  in the boundary of  $F(v_1,v_2,x_1)$ that are guaranteed to exist by 
      Remark ~\ref{2-face-source-sink}.   For $G(P,\bc )$ to be the Hasse diagram of a poset $Q$,   $e_{v_1,v_2}$ must give rise to 
       a cover relation  $v_1 \prec v_2$ in $Q$.  However,   the other directed path from $v_1$ to $v_2$ in the boundary of $F(v_1,v_2,x_1)$ 
       %includes $x_1$ as an intermediate element and 
       gives   a saturated chain from $v_1$ to $v_2$ in $Q$ having $x_1$ as an 
       intermediate element, a contradiction to $e_{v_1,v_2}$ being a cover relation. 
% there is another saturated chain from $v_1$ to $v_2$  in $L$ which has  $x_1$ as an intermediate element, namely the other directed path from 
 %$v_1$ to $v_2$ in the boundary of $F(v_1,v_2,x_1)$.  
% The  existence of this other saturated chain  implies  $v_1\prec v_2$ cannot be   a cover  relation in $Q$, 
  %ing  the definition of cover relation.  This 
 % contradicting our assumption that    $G(P,\bc )$  was  a Hasse diagram.
 %   This completes  the proof.  
\end{proof}

%Lemma ~\ref{not-same-source-sink}  together with reasoning as in its proof
%yields  the following further  result. 

With a little more work we also deduce the following: 

\begin{lem}\label{source-non-revisit}
Let $P$ be   a simple polytope  and let  $\bc $  be a generic cost vector  such that $G(P, \bc)$ is  the Hasse diagram of a poset $Q$.  
Let  $F$ be any face of $P$  containing the source vertex of $P$. %, in other words  containing $\hat{0}\in Q$.   
Then  %  has the property that 
each  edge  $e_{v,w}$ with  $v\in F$ and $w\not\in F$ must be %  from  a vertex in $F$ to a vertex  outside of $F$ is 
 oriented from $v$ to  $w$.  % away from  the vertex in $F$.
Likewise for any face $F'$ of $P$ containing the sink vertex of $P$, any % and each 
edge $e_{x,y}$ with  $x\in F'$ and $y\not \in F'$ must be oriented from $y$ to  $x$.
% from a vertex  in $G$ to a vertex  outside of $G$ must be oriented towards the vertex in $G$.
\end{lem}

\begin{proof}
%Let $S$ be a set of atoms in  $L$.  
%%%the partial order with the Hasse diagram induced on the 1-skeleton of 
%%%$P$ by $\bf{c}$. 
% Let $F_S$ be the smallest face of the polytope which  includes as vertices all of the elements of $S$ as well as  including $\hat{0}$.
Let $F$ be a face of $P$  containing the source vertex of $P$, so in other words containing  $\hat{0}\in Q$.  Suppose there is  an  edge 
 $e_{w,v}$ oriented from $w\in P\setminus F$ to $v\in F$.  Since $\hat{0}$ is the unique source 
 in $P$ (and hence in $F$), there must be a  directed path $\hat{0} = v_0 \prec v_1 \prec v_2 \prec \cdots \prec v_k 
 \prec v$ staying within $F$.  Now we will apply Lemma ~\ref{not-same-source-sink} to the face $F$ viewed as a codimension one face of the unique simple polytope $G$ containing both  $e_{w,v}$ and $F$. % with $\dim (G) = \dim (F) + 1$.

  The proof of the statement for  any face $F'$ containing the sink  of $P$ is similar. 
\end{proof}

%This directly implies the nonrevisiting property for such faces $F$ and $G$, as discussed next.

\begin{cor}\label{source-face-sink-face}
Let $P$ be a simple polytope and  let  $\bf{c}$ be a generic cost vector  such that $G(P,\bf{c})$ is the Hasse diagram of a poset.  Then each face $F$  of $P$ which contains $\hat{0}$ has the  property that there are no directed paths which  %depart $F$ and later 
revisit $F$.  
Likewise for each  face $G$ containing $\hat{1}$, 
there cannot be  any directed paths that  %depart  $G$ and later 
revisit  $G$.  
\end{cor}

Corollary ~\ref{source-face-sink-face}  has the following important special case. 
% See Definition ~\ref{pseudo-def} for the notions of pseudo-join and pseudo-meet, denoted $psj(S)$ and $psm(T)$ for $S$ (resp. $T$) a set of atoms (resp. coatoms).

\begin{cor}
Let $S$ be any set of atoms in a poset $P$ whose Hasse diagram is $G(P,\bc )$ for $P$ a simple polytope.  Then 
each directed path of $G(P,\bc )$  from $\hat{0} $ to $psj (S)$ 
stays within the unique  smallest  face $F_S$  of $P$ containing $\hat{0}$ and all of the elements of  $S$.    
Likewise for each set $T$ of coatoms, each directed path from
$psm(T)$ to $\hat{1}$ stays within the unique  smallest face $F_T$ containing $\hat{1}$ and 
all of the coatoms in $T$.

\end{cor}

This in turn  implies the following:

\begin{cor}\label{join-inside}
Let $P$ be  a simple polytope  and let  $\bf{c}$ be a generic cost vector such that $G(P,\bf{c})$ is the Hasse diagram of a lattice $L$. 
Then the join  of any collection $S$ of atoms in $L$  is contained in the unique  smallest face $F_S$ containing $\hat{0}$ and all of  the atoms in $S$.  
Dually,  the meet of any collection $T$ of coatoms in $L$ is contained in the unique  smallest face $G_T$ containing  $\hat{1}$ and all of the coatoms in $T$.  
\end{cor}

\begin{proof}
%By definition of source and sink, there is a directed path from any atom in $S$ to $psj(S)$ that stays within $F_S$.  
There is a directed path from any  atom $a\in S$ to the join    $J(S)$ of the set of elements of  $S$ since $J(S)$ is an upper bound for the elements of 
$S$.   %$\vee_{a\in S} a$ 
%by the  definition of upper bound. 
%the join $J(S)$ of the set of atoms in $S$, 
%due to the least upper bound of a set of atoms being  an upper bound of this set of  atoms.  
There is also a directed path from % $\vee_{a\in S} a$  
$J(S)$ to  the pseudo-join $psj(S)$ by virtue of $psj(S)$ being an upper bound for the elements of $S$, hence being greater than or equal to  the least upper bound 
 $J(S)$.  % for the elements of $S$.   
  Concatenating these directed paths yields  a directed path $p$ from $a$ to $psj(S)$.  By Corollary ~\ref{source-face-sink-face}, 
   %Lemma ~\ref{source-non-revisit},  
   $p$ must stay within $F_S$ since
  $\hat{0}\in F_S$.  In particular, this  implies $J(S) \in F_S$.   The proof for coatoms is similar.
  % also follows by  replacing the cost vector with its negation and noting that 
 % $G(P,-\bf c) $  is the Hasse diagram for the dual poset to the poset having Hasse diagram  $G(P,\bf{c})$.  
\end{proof}

%One might be tempted to try to 
%generalize Lemma ~\ref{source-non-revisit}  %and  many of its corollaries  
%by replacing $\hat{0}$ % and a set of atoms 
%by  an arbitrary element  $u\in L$  while 
%replacing the atoms of $L$  by the atoms of a closed interval $[u,v]$ in  $L$.  However,  our 
%proof of Lemma ~\ref{source-non-revisit}  does not apply for  $u\ne \hat{0}$ due to $G(P,\bc )$ having one or more edges directed towards any  
%$u\ne \hat{0}$. 
% Specifically, there will be  edges from $u\in F$ to vertices outside $F$  with these edges directed towards $u$ in $G(P,\bc )$, since $u\ne \hat{0}$, rendering the proof of Lemma ~\ref{source-non-revisit} inapplicable. 
  %  Nonetheless, with some  effort we will  obtain  weaker results seemingly 
   % heading towards analogues of  Lemma ~\ref{source-non-revisit} and 
   % Corollary ~\ref{source-face-sink-face}  for arbitrary intervals $[u,v]$.

Next is a result %, Theorem ~\ref{2-atom-base-case},
 %regarding when  $psj(x,y)$ % the pseudo-join of a pair of  elements $x,y$  
% equals $x\vee y$ %, denoted $psj(x,y)$, 
%together with  a dual result regarding when $psm(x',y')$ %  the pseudo-meet of elements $x',y'$ 
%equals $x'\wedge y'$.  %, denoted $psm(x',y')$.  
%This 
that will serve as  a key ingredient 
%the base case  %for the induction 
in the   proof of  Theorem ~\ref{pseudo=join}.   The proof of Theorem ~\ref{2-atom-base-case} 
 is one of  the trickiest proofs in the  paper, 
 % (other than perhaps the 
%appendix).
so we give an especially high level of detail in the argument below.
% in order  to help readers who wish to check each step carefully.  
%,  in which  
%the 2-faces below are replaced more generally  with $r$-faces for any $r\ge 2$. 

%%%%%%%%%%%%%%%%%%%%%%%%%%%%%%%%%%%%%%%%%%%%%%%%%%%%%%

\begin{thm}\label{2-atom-base-case}
Let $P$ be a simple polytope and  let  $\bf{c}$ be  a  generic cost vector  such that  $G(P,\bf{c})$ is the Hasse diagram of a lattice $L$.  
Let $F$ be a 2-face  in $P$, let $u$ be  the source vertex  in $F$, and let $x, y \in F$ be vertices both 
 covering $u$ in $L$. 
 % This implies that the  element   $psj(x,y)$ which is  the sink of $F$ satisfies
 Then   $psj(x,y) = x\vee y $.    Dually,  for any 2-face $F'$ in $P$  with sink $v$ and elements  $x',y' \in F'$ both covered by $v$ in $L$, 
  %$F'$ has the property that 
 $psm(x',y') = x'\wedge y' $.  %In particular, $x\vee y\in F$ and $x'\wedge y' \in F'$.
\end{thm}

\begin{proof}
If   a 2-face $\sigma $ in $P$  with source $u$ has  
$x, y\in \sigma  $ satisfying $u\prec x$ and $u\prec y$  in $L$ with  $psj(x,y) \ne x\vee y$, let us say that $\sigma $ has property (LT). 
  Since $psj(x,y)$ is an upper bound for $x$ and $y$, property (LT)  implies 
$x\vee y <_L  psj(x,y)$ where $<_L$ denotes the order relation for  $L$. %, hence the name (LT) which is short for ``less-than''.  
% or  that there exists a 
%2-face $F'$ in $P$  with sink $v$ having 
%$x', y'\in F'$ satisfying $x'\prec v$ and $y'\prec v$   in $L$  with 
% $psm(x',y') \ne x' \wedge y'$.   
% Let us refer to any 2-face $F$ with source $u$  satisfying  this condition  %the former  condition 
% as having {\bf property (LT)}.  % In the former case, 
%%%  Given a  face $F$ with $x,y\in F$  having $x\vee y \ne psj(x,y)$ and with the source of $F$ being 
%% %$u$,
%%%  For $u$ the source of a face $F$
 % In the former case,  l
 Given any  $u\in L$, let $d(u)$ be 
%%% choose such a face $F$ or $F'$  within $P$ having
 the length of the longest saturated chain in $L$  from $\hat{0}$ to $u$. %  (for such  $F$)
%%% In the latter case, 
%% %Given  a  face $F'$  with $x',y'\in F'$ having $x' \wedge y' \ne psm(x',y')$ and with the sink of $F'$ being $v$, 
%%%For $v$ the sink of a face $F'$, 

We will prove below  that  $P$ cannot have any  2-faces with property (LT).    %this  cannot actually occur, 
%yielding the first half of the desired result.   
The dual case will then  follow  %from this case 
%may be reduced to the former case %  then will follow
 by replacing the cost vector $\bc $ with its negation $-\bc $ which will reverse the direction of all the directed edges in $G(P,\bc )$, 
 having the effect of replacing $L$ by the dual poset $L^*$ with  the meet  operation for $L$ becoming  the join operation for $L^*$ and the  pseudo-meet operation in $G(P,\bc )$ becoming the  pseudo-join operation  in $G(P,-\bc )$.  
%Thus, it suffices to rule out the former case  in order to rule out both.  
% In the latter case, 
 %let $d(F',v)$ be  the length of the longest saturated chain  in $L$  from  
%$v$  to $\hat{1} $.    
%%%Letting $l$ denote the length of the longest saturated chain in $L$, we 
%%%will prove nonexistence of all such $F$ and $F'$ by an argument that could be framed as induction on $l - d$ where $d$ equals  $d(F,u)$ in the former case or $d(F',v)$ in the
%%%latter case.  depending whether 
%%%we are trying to rule out existence of 
%
%???? be careful below that we use $d(F,u)$ in the join case and $d(F',v)$ in the meet case.
%
%??? can we just prove the join case and then say to dualize everything for the meet case to simplify various bits
%
%Consider

 Supposing  there were   a   2-face % of one of the two types above, namely either
 with property (LT), % the join of its atoms not equalling the pseudo-join of its atoms, 
 % or having the meet of its coatoms not equalling the pseudo-meet of its coatoms; more specifically, choose
 choose  such a 2-face  $F$ with source $u$  %  $F$ in a way that 
%$F$ or $F'$ as above which  
with   $d(u)$ as large as possible among all 2-faces with property (LT).
%achieves  the overall largest  possible value % $d^{\max }$
%achieved anywhere within $P$ by  % either of the two 
%the quantity $d(F,u)$  %and $d(F',v)$ 
%among faces satisfying property (NE).    
Denote this largest value for  $d(u)$   by $d^{\max }$. 
%for either  $d(F,u)$ or $d(F',v)$.  % occurring  in $P$.
%  makes  this  associated 
%length  the absolute largest possible
 %among all possible choices of either type (namely over all choices of $F$ or $F'$)
  %for our given polytope $P$.
%For convenience in notation, let us assume  that $d^{\max } = d(F,u)$  for a 
%2-face $F$ having elements  
%$x,y$ covering the source  $u$  in  $F$ with 
%$x\vee y \ne psj(x,y)$; the case  instead having  
%$d^{\max } = d(F',v)$ for $v$ a sink in $F'$ is  entirely analogous with everything dualized. 
%%%  Our proof will essentially be by induction on the $l-d^{\max }$ where $l$ is the 
%%%length of the longest saturated chain from $\hat{0}$ to $\hat{1}$.  
%In what follows, we will utilize the nonexistence of any 2-face $F''$ with source $u''$ having $d(F'',u'')>d^{\max }$ where $F''$  has elements  $x'', y''$ both covering 
%$u''$ whose pseudo-join does not equal their join; w
%In what follows, w
We will  show    %give an  inductive sort of proof, showing
 that the existence of such a 2-face  $F$ forces the existence 
  %so as to get a contradiction to existence 
  of such a 2-face  $F''$  with source $u''$ also satisfying property (LT)  with $d(u'')>d^{\max }$, thereby yielding  a contradiction.
  Figure 1 may help in reading this upcoming argument. % to the existence of a 2-face having property (LT).
%Our construction of such $F''$ will rely on  the observation  that the function $d$ satisfies the following monotonicity property: 
 %for $u_1 < u_2$  in $L$  %(resp. $v_1 < v_2$) 
 %with $u_1$ % (resp. $v_1$) 
  %the source % (resp. sink)
  % of  a 2-face 
%$F_1$ and $u_2$  %(resp. $v_2$) 
%the source % (resp. sink)  
%of a 2-face 
%$F_2$, the existence of a saturated chain  from $u_1$ to $u_2$ implies $d(F_1,u_1) < d(F_2,u_2)$. % (resp. $d(F_1,v_1) > d(F_2,v_2)$).  

%Given  a  2-face $F$ with  $x, y\in F$ both covering the source $u$ in $F$, with   $x\vee y \ne psj(x,y)$, and with  $d(F,u) = d^{\max }$, 
Let $x$ and $y$ be the elements  of $F$ covering $u$ in $L$.     %See Figure 1.  
The Hasse diagram of $F\cap (u,psj(x,y))$ is obtained by removing the vertices $u$ and 
$psj(x,y)$ from the boundary of $F$;   %by  Remark ~\ref{2-face-source-sink}, 
this  yields two disjoint nonempty paths, one containing $x$ and the other containing $y$ (by  Remarks ~\ref{face-unique-source-sink} and  ~\ref{2-face-source-sink}).  
Denote the vertices proceeding upward in $L$ along one of these two  paths 
from $u$ to $psj(x,y)$ 
%in  the boundary of  $F$ 
as $u, x_1, x_2,\dots , x_r$ where  $x_1=x$ and $x_r = psj(x,y) $.  
Denote the vertices proceeding  upward in $L$ along the other  path from  
%directed path in the boundary of  $F$ from 
$u$ to $psj(x,y)$ as 
$u, y_1,y_2,\dots ,y_s$ where $y_1 = y$ and $y_s = psj(x,y)$.  
% Do this in such a way that we have  $x = x_1, y = y_1$ and $x_r = y_s=psj(x,y)$.

The strict  
inequality $x\vee y <_L psj(x,y)$  implies either (a)  $x\vee y\not\in F$ or (b)  $x\vee y \in \{ x_1,\dots ,x_{r-1},y_1,\dots ,y_{s-1}\} .$  In the latter case, we may assume 
without loss of generality that $x\vee y\in \{ x_1,\dots ,x_{r-1}\}$.  Let us now show that in either case there must be some directed path that exits $F$ at some $y_i$ with $1\le i < r-1$ and revisits $F$.  This follows in case (a) from the existence of a directed path from $y = y_1$ to $x\vee y\not\in F$ along with the existence of a directed path from $x\vee y $ to $psj(x,y)$, the latter of which exists because $psj(x,y)$ is an upper bound for  $x$ and $y$ and hence  must be greater than or equal to $x\vee y$.  In case (b), we use the existence of  a directed path from $y = y_1$ to $x\vee y\in \{ x_1,\dots ,x_{r-1}\}$ together with the  disconnectedness of the restriction of $F$ to the vertex set $\{ x_1,\dots ,x_{r-1},y_1,\dots ,y_{s-1}\} $; this directed path  from $y_1$ to $x\vee y$ must  exit and reenter $F$ to get from the connected component containing $y_1$ to the  connected component containing $x\vee y$.  In either case, we have now shown that  there is a directed path  that exits $F$ at some $y_i$ for $i\in \{ 1,2,\dots ,s-1\} $  and revisits $F$. % justifying our claim in case (b).  

       \begin{figure}[h]\label{2-face-picture}
        \begin{picture}(325,312)(-80,170)
 
 \put(65,320){$F$}
 \put(172,315){$G$}
 
  \put(70,462){$v=psj(x,y)$}
       \dashline{3}(76,458)(107, 420)%{\line(1,-1){36}} %y_k to v
       \curve(70,458, 28,420)%{\line(-1,-1){36}}
  \put(20,410){$x_{r-1}$}
  \put(107,410){$y_{s-1}$} %y_{s-1}

\dashline{3}(120,405)(133,360)

  \put(0,360){$x_3$}

% \put(40,445){$\hat{1} = $}
%       \put(95,350){\circle{100}}      
 %      \curve(7,372,  20,405)
  \dashline{3}(9,372)(20,405)     
%      \put(80,377){\line(1,-1){27}}
%\curve(77, 356, 94, 361, 111,377)
%        \put(-17,349){\circle*{4}}
%      \put(183,349){\line(1,1){28}}
        \curve(5, 357, 5,300) 

\put(0, 290){$x_2$}
\curve(7,283, 23,230)
\put(18,220){$x = x_1$}

\curve(28, 215, 65, 185)
\put(70, 180){$u$}

\curve(80, 185, 110, 215)
\put(110, 220){$y = y_1$}

\dashline{3}(118,228)(135,255)
%\curve(-7,66, -13,66, -12,59)

%\put(92,-25){$\hat{0}$}
          
   \put(135,260){$y_i$}
   \curve(138, 267, 140, 300)
   \put(137, 305){$y_{i+1}$}
   \curve (142, 267, 190, 300)
   \put(195, 303){$z$}
   
   \curve(140, 313, 137, 345)
   \put(133, 350){$y_{i+2}$}
   
   \curve(144, 313, 170, 340)
   \put(170, 345){$z^+$}
   
   \dashline{3}(201,312)(223, 352) %z to psj(z,y_{i+1}) = z\vee y_{i+1}
   \put(218, 360){$psj(z, y_{i+1}) = z\vee y_{i+1}$}

  \dashline{3}(220,370)(86,455) %psj(z,y_{i+1}) to top
   
   \dashline{3}(180, 347)(212, 355) %z^+ to psj(z,y_{i+1}) 
   
%   \dashline{3}(197, 312)(185, 365)
%     \dashline{3}(192, 308)(-50, 365) % z to x\vee y
 %  \put(180, 370){$x\vee y$}
% \put(-68,370){$x\vee y$}
  % \dashline{3}(175,365)(26,228)  %x_1 to x\vee y
 %      \dashline{3}(-58,365)(16,228)  %x_1 to x\vee y
%  \dashline{3}(-55, 380)(62, 455)      %x\vee y to psj(x,y)
%  \dashline{3}(-38, 375)(62, 403)      %x\vee y to z'
%\dashline{3}(77,407)(103,410) %z' to y_k
%\put(66,403){$z'$}

% \curve(187, 385, 35, 408)
        \end{picture}
        \caption{Diagram for proof of Theorem ~\ref{2-atom-base-case}}   
\end{figure}

  %%%%%%%%%%%%%%%%%%%%%%%%%%%%%%%%%%%%%%%%%%%%%%%%%%%%%%%%%%%%%%%%%%%%%%%%%%%
  
% To complete the proof  we next
 %We will spend the remainder of the proof ruling  out  the existence of such directed paths that exit and reenter $F$.  %(b). % cannot happen.  
 %Suppose there is 
 Consider  a   directed path $p_{i}$ %,k}$ % from $y'$  %(i.e. from $y_m$ above) 
%  to $psj(x,y)$  
 that exits from $F$ via  an edge from some  $y_i\in F$  %with $i\ge m$
  to some $z\not\in F$ and later  revisits  $F$. % at some $y_k$ or some $x_k$.
  % via  an edge from  some  $z'\not\in F$ either  
 %to some  $y_k\in F$ with 
 %$k\le s$  or to some $x_k\in F$ with $k\le r$.  %, choosing $y_i$ and $y_k$ in a way that minimizes 
Choose such  $y_{i}$ and  $p_i$   %,k}$    
with  $i$ as large as possible.  See Figure 1. %$k-i$   as small as possible.  
%We will  focus on the case with $p_{i,k}$ reentering at  $y_k\in F$ for some $k\le s$, since  %, handling this case in such a manner that the 
%the case with $x_k\in F$ and $k\le r$ is  proven entirely analogously. 
%By our set-up,  $y_k$ % (resp. $x_k$) 
 %is an upper bound for $y_{i+1}$ and 
 %$z$, which implies that  $y_{i+1}\vee z \le y_k$. % (resp. $y_{i+1}\vee z \le x_k$).  
 Our choice of   $F$ with source $u$ such that   $d(u) = d^{\max }$ %, namely with $d(F,u)$ as large as possible, 
 % inductive hypothesis regarding the length of the longest path  from $\hat{0} $ to  the source of our 2-face  $F$  
 allows us to deduce that  $y_{i+1}\vee z = psj(y_{i+1},z)$ as follows: otherwise  $y_i,y_{i+1}$ and $z$ would all belong to  a 2-face having source $y_i$ and having
 property (LT) with  $d(y_i) >  d(u)$,  % by virtue of having $u<y_i$; %. %, using  
 %the monotonicity property discussed above.  
 %this would contradict 
 contradicting our choice of $F$ and  $u$ with $d(u)$ as large as possible.
 %; this would  contradict our assumption that $d(F,u)$ maximized this quantity among faces having their  
 %pseudo-join of atoms  not equalling their  join of  atoms.  %  cannot give rise to another such
 %$F$.  
%
% Thus we may conclude
Thus, we may assume  $y_{i+1}\vee z = psj(y_{i+1},z)$.  
But  the identity  $y_{i+1}\vee z =psj(y_{i+1},z)$  implies that   %is  in turn   %But  knowing $y_{i+1} \vee z = psj(y_{i+1},z)$
   %implies that 
    the four vertices 
 $y_i, y_{i+1}, z$ and $y_{i+1}\vee z$ are all contained  in a single 2-face  $G$  % of $P$ 
 having $y_{i+1}\vee z$ as 
 its sink.    %  We also must have
 %  We must have  $G \ne F$, 
 Since $z\in G$ and  $z\not\in F$, we must have $G\ne F$.   %Distinctness of % the 2-faces 

  Since $F$ and $G$ are distinct  2-faces, they  intersect in a proper face %, they % implies that
   %$F$ and $G$ 
    %intersect in 
    which is  at most an edge.   Thus,  they  share at most two vertices.   
 %By virtue of how they are defined, 
 $F$ and $G$ do share $y_i$ and $y_{i+1}$, hence cannot share any further vertices.  %, so cannot share any other vertices; 
 Since 
 $y_{i+1}\vee z \in G$  %and $y_k \in F$ % (resp. $x_k\in F$) 
  for $y_{i+1}\vee z \not\in \{ y_i,y_{i+1}\} $,  % and $y_k\not\in \{ y_i,y_{i+1}\}$,  %(resp. $x_k\not\in \{ y_i,y_{i+1}\}$), 
% the observation that $F$ and $G$ share at most two  vertices allows us to 
 it follows   that 
 %$y_{i+1}\vee z  \not \in F$, hence 
 $y_{i+1}\vee z \not \in F$.  % \ne y_k$. % (resp. $y_{i+1}\vee z\ne x_k$).  % since otherwise $G$ and $F$ would share 3 distinct vertices, namely 
 %$y_i, y_{i+1}$ and $y_k$, forcing $G$ and $F$ to be the same 2-face as each other.  
 %The fact that  we have deduced the existence of 
 %There is by construction a directed path from $y_{i+1}$ to $y_k$. % (resp. $x_k$)  
 %We may use % also get a directed path  from $z$ to $y_k$,  % (resp. $x_k$), 
 %use a portion of the path  $p_{i,k}$ to obtain a directed path   path from $z$ to $y_k$, implying $z\le y_k$. % (resp. $x_k$), 
%Since $y_{i+1}\le y_k$ as well, we deduce that
%%%% This shows that   $y_k$  %(resp. $x_k$) 
 %%%%is an upper bound for $y_{i+1}$ and $z$, implying 
 %$y_{i+1}\vee z\le y_k$ by the definition of the join operation. % (resp. $y_{i+1}\vee z\le x_k$).  
 %This  implies that
 Similarly  the elements of $G\cap (y_{i+1},y_{i+1}\vee z]$ are not in $F$, so in particular the element  $z^+$   covering $y_{i+1}$ in   $G$ is not in   $F$. 
  But we have a directed path from 
 $z^+$ to $y_{i+1}\vee z$ by virtue of  $z^+$ being  in the 2-face $G$ with $y_{i+1}\vee z$ % (which equals $psj(y_{i+1},x)$) 
  as its sink.  Thus,   $z^+ \le_L y_{i+1}\vee z$.  
 We also have $y_{i+1}\vee z \le_L  psj(x,y)$ because  $psj(x,y)$ is   an upper bound for  $y_{i+1}$ and $z$.  % the  definition of the join operation. 
 Combining inequalities yields $z^+ \le_L psj(x,y)$.   Thus, we  %have a  directed 
 %edge exiting $F$ at $y_{i+1}$ going to $z^+\not\in F$.  We  
 %also 
 have  a directed path from $z^+\not\in F$ to a vertex $psj(w,y) \in F$. %, by virtue of having $z^+\le psj(x,y)\in F$.  % with $psj(x,y) \in F$). 
Since we also have a directed edge from $y_{i+1}$ to $z^+$,  this  directed edge and this directed path % may be concatenated 
together  % to 
exhibit  the existence of a directed path exiting $F$ at $y_{i+1}$, going through $z^+\not\in F$, and later revisiting  $F$.  This contradicts 
our  choice of $y_i$ with $i$ as large as possible.
\end{proof}

%Next   we generalize Theorem ~\ref{2-atom-base-case} to larger collections of elements all covering an element $u$.  

%to the more  general  result.
%, namely the analogous result for  joins of potentially more than two elements  all covering a fixed element: 
% for  %, namely with %  $x$ and $y$ replaced 
% by 
%a potentially
 %larger collection $\{ a_1,\dots ,a_d\} $ of 
%elements all covering a common element $u$.

\begin{thm}\label{pseudo=join}
Let $P$ be a simple polytope and let $\bc $ be  a generic cost vector  such that  $G(P,\bc )$  is the Hasse diagram of a  lattice $L$.  
Given any   $u\in L$ and any  $a_1,\dots ,a_j\in L$  all covering $u$ in $L$, $psj(a_1,\dots ,a_j) = a_1\vee \cdots \vee a_j$.  Dually, 
for  any $v\in L$ and any $c_1,\dots ,c_j\in L$  all covered by $v$ in $L$, $psm(c_1,\dots ,c_j) = c_1\wedge \cdots \wedge c_j$.
\end{thm}

\begin{proof}
%The proof is by induction on $j$.   
The $j=1$ case holds tautologically.  The $j=2$ case is proven  in Theorem 
~\ref{2-atom-base-case}.  We will rely on these  cases to prove the result for any fixed  $j>2$ (though the proof  for $j>2$ is not by induction on $j$). 
% These  two cases together  provide the base case for the  proof by induction.
% to deduce the result for progressively larger $j$ (whereas $j=1$ alone would not).
 % The case with $j=1$ holds  tautologically. %
 %In what follows, l
 Given $u\in L$ and $ \{ a_1,\dots ,a_j \} $ a set of elements which all cover  $u$ in $L$, recall our notation 
 $J(a_1,\dots ,a_j)$ for   $a_1 \vee \cdots \vee a_j$.  %.  For $S = \{ a_1,\dots ,a_j\} $ this same set of elements all covering $u$,  
%2023: commented out next line as defined earlier. 
%Let $J(S)$ denote the join of the elements in a set $S$.   
% Let $F_{\{ a_1,\dots ,a_j\} }$ be the unique smallest face containing $a_1,\dots ,a_j$.   
 Let $F_{ \{ a_1,\dots ,a_j\} }$ be the unique smallest face containing  all of the elements in  %S just changed to \{ a_1,\dots ,a_j\} 
 $\{ a_1,\dots ,a_j \}  \cup \{ u\} $.  Thus, $u$ is the source of $F_{\{ a_1,\dots ,a_j\} }$.
 %  for $S$ a set of elements  of $L$  all covering some element  $u\in L$.
  %; in  this case,  $u$ is  the source of $F_S$.  
 Let $\le_L $ denote the order relation on the lattice $L$.

  %To prove the result for   $j\ge 3$, we inductively assume 
  %for all   $u\in L$  and all   $T \subseteq \{ a_i \in L | u\prec a_i\} $ satisfying   $|T| < j$
  %that $J(T) = psj(T)$.  
  %To prove $J(a_1,\dots ,a_j) = psj (a_1,\dots ,a_j)$, we will  prove $psj(a_1,\dots ,a_j) \le_L J(a_1,\dots ,a_j)$ and
 % $J(a_1,\dots ,a_j)\le_L  psj(a_1,\dots ,a_j)$.    We begin with the much more challenging direction, namely  proving 
  First we  will  prove $psj(a_1, \dots ,a_j)\le_L J(a_1,\dots ,a_j)$.
  %since this will be the much more  challenging direction.   
 To accomplish this, % our plan is to
 we will  prove for each $x\in F_{\{ a_1,\dots ,a_j\} }$ other than $psj(a_1,\dots ,a_j)$ 
  that each $y\in F_{\{ a_1,\dots ,a_j \} }$ that covers $x$ satisfies $y\le_L J(a_1,\dots ,a_j)$, assuming  inductively that 
  we have already proven this same 
  statement for every $x'\in F_{\{ a_1,\dots ,a_j\} }$ such that $x' <_L x$.  
  To get started, we first   check  the base case of this  inductive statement, namely the 
  case of $x=u$.  The definition of the 
  join operation directly yields $a_i\le_L J(a_1,\dots ,a_j)$ for $i=1,\dots ,j$.  This in turn implies $u\le_L J(a_1,\dots ,a_j)$ since $u\le_L a_i\le_L J(a_1,\dots ,a_j)$ for 
 $i=1,2,\dots ,j$.    This completes the  base case.

       \begin{figure}[h]\label{higher-picture}
        \begin{picture}(325,292)(-110,170)

\curve(28, 215, 65, 185) %u to a
\put(70, 180){$u'$}

\curve(80, 185, 110, 215) %u to a'

 \put(18,220){$a$}
\put(110, 220){$a'$}
\dashline{3}(25,228)(107,405) %a to psj(a,a')

\dashline{3}(107,227)(-50,445)
 \put(-68,450){$J(a_1,\dots ,a_j)$}
  % \dashline{3}(175,365)(26,228)  %x_1 to x\vee y
 \dashline{3}(-58,445)(16,228)  %x_1 to x\vee y
 
 \dashline{8}(-40,445)(103,413) %J to psj
 
 \curve(24,379,34,402) %implies symbol
 \curve(29,377,39,399) %other half
\curve(28,405,41,409) %left tip
\curve(46,398,41,409) %right tip

  \put(77,303){$F(x,y,w)$}
\dashline{3}(118,230)(135,255) %a' to w
   \put(135,260){$w$}

   \curve(138, 267, 140, 300) %w to x
   \put(137, 305){$x$}

  \curve(140, 313, 137, 345) %x to y
   \put(133, 350){$y$}

  \put(107,410){$psj(a,a') = a\vee a'$} %y_{s-1}
\dashline{3}(120,405)(133,360) %y to a\vee a'

%  \put(70,462){$v=psj(x,y)$}
%       \dashline{3}(76,458)(107, 420)%{\line(1,-1){36}} %y_k to v
 %      \curve(70,458, 28,420)%{\line(-1,-1){36}}
 % \put(20,410){$x_{r-1}$}

 % \put(0,360){$x_3$}

% \put(40,445){$\hat{1} = $}
%       \put(95,350){\circle{100}}      
 %      \curve(7,372,  20,405)
 % \dashline{3}(9,372)(20,405)     
%      \put(80,377){\line(1,-1){27}}
%\curve(77, 356, 94, 361, 111,377)
%        \put(-17,349){\circle*{4}}
%      \put(183,349){\line(1,1){28}}
%        \curve(5, 357, 5,300) 

%\put(0, 290){$x_2$}
%\curve(7,283, 23,230)

%\curve(-7,66, -13,66, -12,59)

%\put(92,-25){$\hat{0}$}
          
 % \curve (142, 267, 190, 300)
  % \put(195, 303){$z$}
   
%   \curve(144, 313, 170, 340)
 %  \put(170, 345){$z^+$}
   
%   \dashline{3}(201,312)(223, 352) %z to psj(z,y_{i+1}) = z\vee y_{i+1}
 %  \put(218, 360){$psj(z, y_{i+1}) = z\vee y_{i+1}$}

%  \dashline{3}(220,370)(86,455) %psj(z,y_{i+1}) to top
   
  % \dashline{3}(180, 347)(212, 355) %z^+ to psj(z,y_{i+1}) 
   
%   \dashline{3}(197, 312)(185, 365)
 %    \dashline{3}(192, 308)(-50, 365) % z to x\vee y
%  \dashline{3}(-55, 380)(62, 455)      %x\vee y to psj(x,y)
 % \dashline{3}(-38, 375)(62, 403)      %x\vee y to z'
%\dashline{3}(77,407)(103,410) %z' to y_k
%\put(66,403){$z'$}

% \curve(187, 385, 35, 408)
        \end{picture}
        \caption{Inductive step for  Theorem ~\ref{pseudo=join}}   
\end{figure}

 Now to the inductive step where we  assume $x>u$. See Figure 2.   Consider any $y\in F_{\{ a_1,\dots ,a_j\} }$ which covers $x$.  
There must also exist  some $w\in F_{\{ a_1,\dots ,a_j\} }$ covered by $x$, since  we already have   %  are beyond the base case which means we have 
%already handled  $x=u$ so may assume
$x>u$.  Consider the unique 2-face $F(x,y,w)$ containing
 $x,y$ and $w$, a face which is guaranteed to exist because $P$ is simple.  
 Let $u'$ be its source.  Since $u' <_L x$, our inductive hypothesis allows us to  assume  % will have already shown
  $u'\le_L J(a_1,\dots ,a_j)$ and that % the same for 
 each  element of $F_{\{ a_1,\dots ,a_j\} }$ covering $u'$ is also bounded above by $J(a_1,\dots ,a_j)$.  In particular, we  may assume % will have already shown
  that the two elements $a,a'$  of 
 $F(x,y,w)$ covering $u'$ % the source vertex of $F(x,y,w)$  
 are  bounded 
 above by $J(a_1,\dots ,a_j)$.  But $psj(a,a') = a\vee a'$  % the join of these two atoms,
 by Theorem ~\ref{2-atom-base-case}.  Thus, $psj(a,a')  \le_L J(a_1,\dots ,a_j)$ due to $J(a_1,\dots ,a_j)$ already being  
 an upper bound for $a$ and $a'$ and hence for $a\vee a'$ as well.  
 The inequality $psj(a,a')\le_L J(a_1,\dots ,a_i)$ implies that every element of $F(x,y,w)$ is  bounded above by
 $J(a_1,\dots ,a_j)$, so in particular that   $y\le_L J(a_1,\dots ,a_j)$.   
 This completes the  inductive step, hence the proof  that %of the inductive step.  Thus, we have shown   
 $psj (a_1,\dots ,a_j)  \le_L J(a_1,\dots ,a_j) $.  

We deduce the inequality 
$J(a_1,\dots ,a_j)\le_L  psj(a_1,\dots ,a_j)$ from the fact that    $psj (a_1,\dots ,a_j)$ is  an upper bound
for  all of the elements $a_1,\dots ,a_j$.
% while  $J(a_1,\dots ,a_d)$ is  the least upper bound for this same set of lattice elements.
%, we  deduce that 
%$psj(a_1,\dots ,a_j) \le_L 
%$J(a_1,\dots ,a_j) \le_L  psj(a_1,\dots ,a_j)$.  
These  weak  inequalities $$psj(a_1,\dots ,a_j) \le J(a_1,\dots ,a_j) \le psj(a_1,\dots ,a_j)$$ combine to yield
% shows that both weak inequalities must be equalities, 
%yielding % we have  now  obtained   yields %above yields % implying 
%Combining the two inequalities yields 
$ psj(a_1,\dots ,a_j) = J(a_1,\dots ,a_j)$, as desired.

This also allows us to deduce   the  desired dual statement 
%equality  of the meet and the pseudo-meet for any  set  $\{ c_1,\dots ,c_j \} $ of elements  all  covered by  a single element $v$ 
as follows.  
We negate the cost vector and observe that $G(P,-\bc )$ is the Hasse diagram of the dual poset to the poset we get from $G(P,\bc )$.  
Poset duality reverses the 
roles of the  meet and join operations, reducing the desired statement about meets to the statement we have already proven about joins.
% to the case of joins and pseudo-joins already handled 
%above.  % allowing  precisely the same proof to go through.
\end{proof}

A useful  immediate consequence of Theorem ~\ref{pseudo=join}
is as follows.

\begin{cor}
Let $P$ be a simple polytope and let $\bc $ be a generic cost vector such that $G(P,\bc )$ is the Hasse diagram of a lattice $L$.  If 
$a_1,\dots ,a_j$ all cover  an element  $u\in L$, then  $a_1\vee \cdots \vee a_j$ is in the unique  smallest face of $P$ containing $a_1,a_2,\dots ,a_j,u$.  
Dually, for  $v\in L$  and $c_1\dots ,c_j$ all covered by $v$,  $c_1 \wedge\cdots \wedge c_j$ is in the unique 
smallest face of $P$ containing $c_1,c_2,\dots ,c_j,v$.
\end{cor}

 Next we deduce  a further 
property of pseudo-joins (and pseudo-meets)  that will be helpful for understanding the topological structure of  posets whose Hasse diagram is 
$G(P,\bc )$ for a simple polytope $P$ and a generic cost vector $\bc $. % for $P$ a simple polytope.  

 \begin{lem}\label{new-pseudo-joins}
Let $P$ be  a simple polytope  and let $\bf{c}$ be a generic cost vector  such that $G(P,\bf{c})$  is the  Hasse diagram of a poset $Q$.   If  $S$ and $ T$ are distinct sets of atoms   %(resp. coatoms)
in  $Q$,  %the partial order with the  given Hasse diagram.   % with 
 %$F_S \ne F_T$, 
 then  psj(S) $\ne $ psj(T). % (resp. psm(S) $\ne $ psm(T)). 
 If $S$ and $T$ are distinct sets of coatoms in $Q$, then $psm(S)\ne psm(T)$.    If  $Q$ is 
 %$G(P,\bf{c})$ is the Hasse diagram of
  a lattice, then the same pair of  statements holds for 
 % we further conclude that
  each interval $[u,v] $ in $Q$.
  % likewise has its
 %also satisfies distinctness of the pseudo-joins of the sets of atoms of the interval, again implying the 
% subposet of pseudo-joins of atoms (resp. pseudo-meets of coatoms)  of $[u,v]$   isomorphic to a Boolean lattice $B_{|A(u,v)|}$ for $A(u,v)$ the set of atoms of $[u,v]$ (resp. 
% $B_{|CA(u,v)|}$ for $CA(u,v)$ the set of coatoms of $[u,v]$). 
\end{lem}

\begin{proof}
Given a collection $S$ of atoms in $Q$, once again let  $F_S$ be the smallest face of $P$ 
containing the vertices in  $S\cup \{  \hat{0} \} $.  Such a face  $F_S$ exists because $P$ is simple.  
 %for $S$. %  a collection $S$ of atoms as well as containing $\hat{0}$.  
Note that $S\ne T$ implies $F_S \ne F_T$ since each face 
$F_S$  is itself  simple with $dim (F_S) = |S|$ and with the neighboring vertices to 
$\hat{0}$ in $F_S$  being exactly the elements of $S$.    

Now to our claim about pseudo-joins.   
 First consider 
$F_T$ that is a codimension one face of a face $F_S$, with  both these  faces  including $\hat{0}$. 
Suppose also that both $F_S$ and $F_T$  have the same sink, denoted $v$.    Then there is an  edge directed from a vertex 
 $v_S \in F_S \setminus F_T$ to  $v \in F_T$, by virtue of $v$  being the sink of $F_S$ and having a neighbor in $F_S\setminus F_T$.  
 There is also an edge  directed from $\hat{0}$  to a vertex  $v_S' \in F_S \setminus F_T$, by virtue of $\hat{0}$ being the  source  of  $F_S$ and having a neighbor in $F_S\setminus F_T$.  
There is also a directed path from  $\hat{0}$ to $v$ that stays within $F_T$  because  $\hat{0}$ and $v$ are the source and sink vertices  of $F_T$.
%, again by virtue of $\hat{0} $ being the source.  
Thus, Lemma ~\ref{not-same-source-sink} applies in this case, giving a contradiction to such faces $F_S$ and $F_T$ having the same sink. 

Next we turn to the case of $F_T \subsetneq F_S$  with both faces  containing $\hat{0}$  with $F_T$ of codimension higher than one in $F_S$.  Again suppose
both $F_S$ and $F_T$ have the same sink.  We use the existence of an intermediate  face $F_{T'}$ with $F_T\subsetneq F_{T'}  \subsetneq F_S$ to reduce as follows  to the codimension one case above.  Since  $F_T$ and $F_S$ have the same source and sink,  $F_{T'}$ must also have this same source and sink, enabling us to reduce to the lower codimension case of $F_T\subsetneq F_{T'}$.  Doing this repeatedly yields the codimension one case above, allowing us to rule out this case as well.
% the existence of such $F_T \subseteq F_S$ with $F_T$ having  codimension higher than one in $F_S$ with both faces having  the same source $\hat{0}$ and the same sink as each other.
% having the same sink and both having source $\hat{0}$. 
 % replacing $G$ by a lower dimensional face also containing $F$ in this manner reduces this case to  the case with $\dim (G) = \dim (F) + 1$ that was already handled above.  
 
 Next we reduce the case of any two faces $F_S$ and $F_{S'} $ both having source  $\hat{0} $  and both  having the same sink to the  
 cases already ruled out above. % as follows.   %Since  $F_S\cap F_{S'}  \subseteq F_S$ and s
 Since  the face $F_S \cap F_{S'}$ must also  
 contain $\hat{0}$ as well as containing the  common sink for $F_S$ and $F_{S'}$, this  implies $F_S\cap F_{S'}$ will  
 also have this same  sink.    But $F_S\cap F_{S'}\subsetneq F_S$, allowing us to use the above cases to rule out this case.
% Thus, we may reduce  the case of any $F_S$ and $F_{S'}$ both containing $\hat{0}$ and both having
 %the same sink as each other  to the case of 
 %$F_T  \subsetneq F_S $ already handled above by letting $T = S\cap S'$.  
  This completes the proof of the statement about distinct atoms in $L$.
  
 Turning now to arbitrary intervals $[u,v]$  of $Q$, we  now assume $Q$ is a lattice.  The fact that $v$ is an upper bound for the set 
 $S$  of  atoms of 
 $[u,v]$  implies  that the join of  the elements of $S$ is contained in $[u,v]$.  Theorem 
 ~\ref{pseudo=join}  ensures that this join equals the pseudo-join of the  elements of $S$, hence that the pseudo-join is also in the interval.  
 These observations  allow the   %above 
 argument above to be applied %more generally
  to the  interval $[u,v]$.
 % by applying our above argument  for the case with $u= \hat{0}$ now instead 
 %to the subposet $[u,v] \cap F_S $;  by definition of $F_S$, the face $F_S$ also 
 %will include the pseudo-join of any set of atoms of $[u,v]$, just as needed.
 
 The same proof applied to the dual poset yields  the desired analogous  
 statements for pseudo-meets of coatoms in all of $Q$ as well as  in any interval $[u,v]$.
\end{proof}

\begin{cor}\label{implies-Boolean}
Let $P$ be a simple polytope and let  $\bc $ be  a generic cost vector such that $G(P,\bc )$ is the Hasse diagram of a poset  $Q$.
Then  the subposet of $Q$ consisting of all 
 pseudo-joins of atoms (resp. pseudo-meets of  coatoms)  
 is a Boolean lattice $B_{|A|}$ for $A$ the set of atoms (resp. coatoms) of $Q$.  If $Q$ is a lattice, then this same property holds for each 
 interval $[u,v]$ in $Q$.  
\end{cor}

Now to a topological consequence of  the above results.  
%This will be based on a consequence of the above results, as discussed shortly,
%namely that the subposet of joins of atoms is a 
% Boolean lattice.
 % structure of the subposet of joins of atoms that we have just developed   to give 
%a poset map to a Boolean lattice. % amenable to usage of the Quillen Fiber Lemma.
%, namely to the poset of subsets of $\{ 1,2,\dots ,n\} $ ordered by containment. 

\begin{thm}\label{simple-polytopes}
Let $P$ be a simple polytope and let $\bf{c}$ be  a generic cost vector  such that the directed graph 
$G(P,\bf{c})$ is the Hasse diagram of a 
%2023: changed next word from poset to lattice
lattice  $Q$.  Then each nonempty open interval $(u,v)$ has order complex that is  
homotopy equivalent to a ball or a sphere.  Thus,  $\mu_Q 
(u,v)$ equals   $0, 1, $ or $-1$ for each $u\le v$.
\end{thm}

\begin{proof}
Given any  nonempty open interval $(u,v)$ in $Q$,  we will 
%the point is to 
define a surjective poset map $f: (u,v) \rightarrow B$  for  $B = B_n \setminus \{ \hat{0},\hat{1}\} $ or 
$B = B_n \setminus \{ \hat{0} \} $ where  $B_n$ denotes the   
Boolean lattice  of subsets of $\{ 1,2,\dots ,n\}$ ordered by containment. 
% or 
%$Q = (\hat{0},\hat{1}]$ in $B_n \setminus \{ \hat{0} \} $ and 
%and to 
%We will check the requisite contractibility of fibers  of $f$ needed to use the Quillen Fiber Lemma (see Lemma ~\ref{Quillen-Lemma}).  
%Our plan is to  apply the Quillen Fiber Lemma (see Lemma 
%~\ref{Quillen-Lemma})  using  this poset map $f$.  
%We will then  use the fact that 
For $n\ge 2$, $B_n \setminus \{ \hat{0}, \hat{1} \} $  
has   order complex homeomorphic to the sphere $S^{n-2}$, by virtue of being the barycentric subdivision of the 
boundary of the simplex $\Delta^{n-1}$, while $B_n \setminus \{ \hat{0} \} $ 
%.   The half-open interval  $(\hat{0},\hat{1}]$ of the Boolean lattice % including  all elements except $\hat{0}$ 
%$(\hat{0},\hat{1}]$ 
has contractible  order complex due to having a cone point at $\hat{1}$. % which is  therefore contractible.

Consider the map   $f$ % to be the map %that we will use   
sending  each $z\in (u,v)$ to the pseudo-join of the set of atoms  $a \in [u,z]$ 
% satisfying
%$u\prec  a\le z$.  
%We proved in Theorem ~\ref{pseudo=join} that 
Under our hypotheses, 
 the join of a set of atoms of an interval $[u,v]$  equals the pseudo-join of this same set of atoms, by Theorem ~\ref{pseudo=join}. 
 % under our hypotheses,
  %and likewise for the restriction to any closed interval $[u,v]$.  
  This  allows a reinterpretation of $f$ as sending $z$ to the join of the set of atoms of $[u,z]$, making  it clear % below $z$ makes it easy to see
  that $f$ is a poset map.
We   proved  that the pseudo-joins of distinct  sets of atoms in an interval $[u,v]$ are themselves distinct  in Lemma ~\ref{new-pseudo-joins}.   This yields  the desired Boolean lattice structure on the image of $f$ as  in Corollary ~\ref{implies-Boolean}.
%Our map $f$  is  a poset map,  by %2023 deleted:   the result from 
%Theorem ~\ref{pseudo=join}.
%2023 deleted rest of sentence:  that each pseudo-join of atoms 
%of an  interval  equals the join of the same set of atoms of that interval.  
 The fibers  of $f$
meet the contractibility requirement  for  the Quillen Fiber Lemma by virtue of each fiber  having a unique smallest  %highest 
element, hence a cone point in the order complex of the fiber. %, completing the proof.  

The claim that  $\mu_Q(u,v) \in \{ 0 ,\pm 1 \} $ for each $u\le v$  now 
 follows directly  from Hall's  well-known interpretation for  $\mu_Q (u,v)$  as 
 the reduced Euler characteristic  % for the order complex of the open interval 
 $\tilde{\chi}(\Delta (u,v))$ together with the fact that $\tilde{\chi }(K)=0$  for $K$ a ball and  $\tilde{\chi }(K) = (-1)^d$ for 
 $K$ a $d$-sphere.
 % the reduced Euler characteristic of a ball is 0 while the reduced Euler characteristic of a $d$-sphere is $(-1)^d$.
\end{proof}

%\begin{rk}
%One might be surprised by  the lack of symmetry in the statement of  Theorem  ~\ref{simple-polytopes}, in the sense  that we require  the top element $v$  of the closed interval $[u,v]$  to be a join of atoms in order to get the homotopy  type of a sphere without making an analogous requirement that the minimal element $u$ by a meet of coatoms.

%  In fact, what is going on is that the former requirement  implies the latter and conversely, as we now explain.  Notice for the interval $[u,v]$ that if $v$ is a join of atoms, then  this implies for each coatom $c$ of the interval that there exists an atom $a$ such that $a\not\le c$.  Thus, the meet of all of the coatoms is some element $x$ which cannot satisfy $a\le x$.  Since this is true for every atom, it follows that the meet of the coatoms must equal  $u$.  Switching the role of atoms and coatoms yields the other implication.   
%\end{rk}

 \begin{rk}\label{almost-never-shellable}
 The posets  $Q$ considered in Theorem ~\ref{simple-polytopes}  typically have order complex which is not shellable. 
  To see this, note that a  shelling would force every 2-dimensional face $F$   in our polytope $P$ 
  to have one of the two directed paths from  the source to  the sink in 
  the  boundary of $F$ to have exactly two edges in it.  This follows from the fact that the boundary of $F$ with its source and sink removed would be disconnected which would force one of its two connected components to have order complex that is  0-dimensional in order  for a shelling of the order complex of $Q$  to be possible.
  %, while  the Hasse diagram
 %property  precludes triangular  2-faces. 
 % (a notion appearing e.g. in
 %\cite{BW-on-lex} and \cite{BW}), 
 %by virtue of 
% Remark ~\ref{2-face-source-sink}.   
 %
% While Theorem ~\ref{simple-polytopes} does  imply  that the
 % order complexes of the posets under consideration are homotopy equivalent to balls or 
%spheres,  these spheres  need not be  top dimensional (in contrast to order complexes of shellable, graded posets). 
%, in particular precluding 
%shellability.
 % More precisely, a
   % Even more stringently, a
%   A  pure shelling would necessitate for every 2-face of the polytope that each of these two directed paths jointly comprising the boundary of the 2-face have  
% length exactly 2; in this case, all 2-dimensional faces would need  to be 4-gons.  
  See e.g. \cite{BW} for background on shellability.  
\end{rk}

\section{Applications to various classes of polytopes and regular CW balls}\label{application-section}

Next  we apply our earlier results to several classes of polytopes (and regular CW balls).  
%In Section ~\ref{nonrevis-subsection}, we give two  seemingly important 
%classes of simple polytopes $P$ and generic cost vectors $\bc $ for which we can prove the face nonrevisiting property for $G(P,\bc )$.   In the first of these two results, we prove for $P$ a  3-dimensional simple polytope and $\bc $ a generic cost vector   with $G(P,\bc )$ the Hasse diagram of a lattice that the face nonrevisiting property holds.  In the second result,  we prove that counterexamples of the Hirsch Conjecture of the type constructed by Santos and others so far cannot also be counterexamples to our Conjecture ~\ref{main-conj}.  
%%%% prove for $P$ a simple polytope with each facet having either the vertex $u$ or the vertex $v$ in it, namely for $P$ a simple spindle, that for any generic cost vector $\bc $ such that $G(P,\bc )$ is a Hasse diagram with source $u$ and sink $v$ 
%
%In Sections, ~\ref{example-section} and ~\ref{topol-subsection}, 
%we show how Theorem ~\ref{simple-polytopes} applies to permutahedra, associatehedra and generalized associahedra.  Then in 
%Section ~\ref{zonotope-section}, we show how it applies to  zonotopes.  We address a question of Victor Reiner regarding generalizing our result beyond polytopes and beyond orientations induced by cost vectors in Section ~\ref{posets-from-shellings}.  
Rather than trying to give as
comprehensive a list  of applications as possible, 
we focus   %in this manner  
on some  important  families  where our  theory applies particularly naturally.  

\subsection{Applications to  3-polytopes and to  spindles}
%Classes of polytopes with the face nonrevisiting 
%property} % (and hence a directed diameter upper bound on $G(P,\bc )$ of $n-d$)}
\label{nonrevis-subsection}

%We begin with two  important classes of polytopes.
%  for which the face nonrevisiting property will  follow 
%from our earlier results  in  particularly  direct and  natural seeming ways.  
%Now to simple 3-dimensional polytopes.  

%We begin with 3-polytopes.

\begin{thm}\label{case-of-3-polytopes}
Let $P$ be a simple polytope of dimension 3, and let $\bf{c}$ be a generic cost vector such that 
$G(P,\bc )$  is the Hasse diagram of a lattice $L$.  Then $G(P,\bc )$ has the nonrevisiting property.  That is, any directed path from $u$ to $v$ with $u,v$ both contained in a face $F$ must stay entirely in the face $F$.
\end{thm}

\begin{proof}
Acyclicity of $G(P,\bc )$ implies the nonrevisiting  property for 0-dimensional faces.  
The  fact that $G(P, \bc )$ is a Hasse diagram  by definition  implies the nonrevisiting property for 1-dimensional faces.  
Suppose there is a 2-dimensional face $F$ in $P$ and a directed path in $G(P,\bc )$ 
 that departs $F$ at $u\in F$ and re-enters $F$ at   $v\in F$.  
 By Lemma ~\ref{not-same-source-sink}, there cannot also be a directed path from $u$ to $v$ that stays entirely in  $F$, since $F$ is a 
 codimension one face of $P$.   Let 
 $a_1$ and $a_2$ be the vertices of $F$ that cover the source of $F$, chosen so that there is a directed path in $F$ from $a_1$ to $u$.
 % However, the fact that  there is no directed path from $u$ to $v$ that stays entirely in the 2-dimensional face $F$, implies  that $v$ is an upper bound in the lattice  $L$ for the two elements  $a_1,a_2$  of $F$ that cover the source vertex of $F$, since (without loss of generality)  
 %Without loss of generality, there is a directed path within $F$ from $a_1$  to $u$ as well as 
 We also have a directed path in $G(P,\bc )$ from $u$ to $v$; the fact that there is no directed path from $u$ to $v$ within $F$ implies in this case that there is a directed path within $F$ from $a_2$  to $v$.   The existence of these three directed paths implies  that $v$ is an upper bound for $a_1$ and $a_2$.  
 But we have just proven in Theorem ~\ref{2-atom-base-case} that the pseudo-join of $a_1$ and $a_2$ equals the join of $a_1$ and $a_2$.  Thus, we deduce that the sink of $F$, namely the pseudo-join of $a_1$ and $a_2$,  is less than or equal to $v$ in $L$.  This implies $v$ is the sink of $F$, since $v\in F$.  But this  contradicts the fact  that there is no directed path from $u$ to $v$ in $F$.  Thus, we have a contradiction to the existence of a 2-face $F$ and a path that starts and ends in $F$ 
 but does not stay entirely in $F$.
\end{proof}

Next  we %show how to deduce from  Corollary ~\ref{source-face-sink-face}    that 
turn to the class of polytopes producing all  
known  counterexamples  to the Hirsch Conjecture, namely spindles.
%, namely  the spindles $P$  (see Definition ~\ref{spindle-def}) 
% with cost vector $\bf{c}$ such that  the two distinguished vertices  of $P$ are the source and sink in $G(P,\bc )$,  
%cannot give any counterexamples to our  %main conjecture, namely 
%Conjecture ~\ref{main-conj} by way of the pairs of vertices which cause them to violate the Hirsch conjecture.
See Definition ~\ref{spindle-def} for a review  of  the notion of  a  spindle.  
 
 \begin{thm}\label{santos-okay}
Let $P$ be   a simple $d$-polytope with $n$ facets.  Suppose that $P$  is a spindle with  vertices $u$ and $v$ such that each facet of $P$ contains either $u$ or $v$.  Let $\bc $ be a generic cost vector such that $G(P,\bc )$ is the Hasse diagram of a poset having $\hat{0} = u$ and $\hat{1} = v$.  Then 
$G(P, \bc )$ satisfies the face nonrevisiting property, implying that every directed path from $u$ to $v$ has at most $n-d$ edges.  This gives  an upper bound of $n-d$ on the distance from $u$ to $v$.  
%violating the Hirsch Conjecture by virtue of the distance between its  two distinguished vertices $u$ and $v$   being greater than $n-d$, and let $\bc $ be  any generic cost vector $\bc $ having $u$ as source and $v$ as sink.  Then  $G(P,\bc )$ cannot be the Hasse diagram of a poset, so in particular it  cannot be the Hasse diagram of a lattice.  
\end{thm}

\begin{proof}
%Suppose $G(P,\bc )$ were the Hasse diagram of a poset.  This would have  
%source $\hat{0} = u$ and sink $\hat{1} = v$.  
The definition of spindle ensures that  each facet $F$ of $P$ includes either $\hat{0}$ or 
$\hat{1}$.  But then Corollary ~\ref{source-face-sink-face}  %Lemma ~\ref{source-non-revisit} 
implies that there are no directed paths that depart from $F$ and later revisit $F$ for faces $F$ that include $\hat{0}$ or $\hat{1}$.  Since every facet in the spindle includes either $u$ or $v$ as a vertex, every facet has this nonrevisiting property.  But every  directed path from $u$ to $v$ departs a facet at each step.  
 Since there are only $n$ facets, and $v$ is incident to $d$ facets, there are at most $n-d$ facets that may be departed, hence at most $n-d$ steps in any directed path from $u$ to $v$.   This implies that  the distance from $u$ to $v$ cannot be greater than $n-d$.  
\end{proof}

This result has as a consequence 
that all of the  known counterexamples to the Hirsch Conjecture (to date)   %resulting from a spindle 
% with distance more than $n-d$ between its two distinguished vertices
fail to meet  the hypotheses for Theorem ~\ref{santos-okay}:

\begin{cor}
%Any counterexample to the Hirsch Conjecture resulting from 
Given any simple $d$-polytope with $n$ facets  that is a spindle with vertices $u$ and $v$ such that every facet includes either $u$ or $v$, if the distance from $u$ to $v$  is greater than $n-d$, then there does not exist any generic cost vector $\bc $ such that $G(P,\bc )$ is the Hasse diagram of a poset with $u$ as source and $v$ as sink.
\end{cor}

\subsection{Simple polytopes with Hasse diagrams of well-known lattices arising as their 1-skeleta}\label{example-section}

We now turn to  three well-known families of lattices, namely the
 weak order, the Tamari lattices, and the Cambrian lattices.  Their Hasse diagrams  will arise as $1$-skeleta of  the 
 permutahedra, the associahedra, and the generalized associahedra, respectively.  
 %Both  of these classes of 
 %examples  will meet all of the hypotheses of Theorem ~\ref{simple-polytopes}.
 % (as well as meeting the hypotheses for our earlier results which Theorem ~\ref{simple-polytopes} relies upon). 
% The first is  the weak order, which has the permutahedron as its associated polytope.  The second is the Tamari lattice, with the associahedron as associated polytope.

 \begin{ex}\label{permutahedron-cost-vector}
 The permutahedron  $P_n$  is a simple polytope yielding weak order as follows.  Let 
 $(x_1,\dots ,x_n)$ be a point in $\reals^n$ with distinct coordinates, most typically chosen with 
 $x_i = i$ for $i=1,\dots ,n$.   Recall that 
 $$P_n(x_1, \ldots, x_n)={\rm conv}\{ (x_{\pi(1)}, \ldots, x_{\pi(n)})| \pi \in S_n\}$$ is the canonical $V$-representation for $P_n$.  Two of its  vertices $x_u=(x_{u(1)}, \ldots, x_{u(n)})$ and $x_v=(x_{v(1)}, \ldots, x_{v(n)})$ for $u, v, \in S_n$  are connected by an edge  if and only if $v=u s_i$ for some adjacent transposition $s_i=(i, i+1)$ acting on values. If starting from $x_e=(x_{1}, \ldots, x_{n})$, corresponding to the identity element $e$ of $S_n$ we orient the edges of $P_n(x_1, \ldots, x_n)$ from shorter towards longer permutations, then we obtain the weak order. Thus a cover relation  $u\prec v$, for $v=u s_i$ in weak order means that we  introduced a descent involving values  $i$ and $i+1$ which were in positions $k$ and $l$, $k<l$, in $u$, respectively. Then, taking the linear functional {\bf c}  $= (c_1,\dots ,c_n)\in \reals^n$ to be one with strictly descending  coordinates $c_1 > c_2 > \cdots > c_n$  we obtain that $\bc \cdot x_v- \bc  \cdot x_u=c_k(i+1)+c_l i -c_k i -c_l (i+1)=c_k-c_l>0$. 
 This verifies for each cover relation $u\prec v$ in weak order that $\bc\cdot x_u < \bc\cdot x_v $.  See also Example 3.3 in \cite{BKS}.
 \end{ex}

\begin{ex}\label{associahedron-cost-vector}
The associahedron $A_n$ is another example of a simple polytope with   a generic cost vector
$\bc $ yielding $G(A_n,\bc )$ as the Hasse diagram of a well known poset, namely  the Tamari lattice.  
Consider the presentation for the associahedron  introduced by Loday in \cite{Lo}. 
The vertices of the associahedron are indexed by the  unlabeled,  rooted planar, binary  trees with $n$ leaves and $n-1$ internal nodes (i.e. non-leaf vertices).
  We associate to  each such  tree $t$ the polytope vertex $M(t)\in \reals^{n-1}$ defined as follows.  $M(t) = (a_1b_1,\dots ,
a_ib_i,\dots ,a_{n-1}b_{n-1})$ where $a_i$ is the number of leaves that are left descendants  of the $i$-th internal node $v_i$ of the tree $t$   and $b_i$ is the number of leaves that are  right descendants  of  $v_i$ within the tree $t$.  One may check, for example, that  the associahedron given by trees with 4 leaves has vertices $(3,2,1), (3,1,2), (1,4,1), (2,1,3), $ and $(1,2,3)$.

Turning now to the Tamari lattice, a
 cover relation $u\prec v$ in the Tamari lattice results from   applying a single associativity relation in our rooted, binary, planar  tree regarded as a parenthesization.  Thus, $v$ is obtained from $u$ by replacing $((x, y),z))$ by $(x,(y , z))$  somewhere in the parenthesized expression, with   the objects $x,y,z$ either being individual letters or being   larger bracketed expressions themselves.   Notice that such an operation 
 will have the impact  within Loday's realization of the associahedron 
 of replacing some pair $(a_i,b_i)$ by $(a_i, b_i + b_{i+r})$ and replacing $(a_{i+r},b_{i+r})$ by $(a_{i+r} - a_i, b_{i+r})$    while leaving  all other   $a_j , b_j$  unchanged. 
   Thus,  $M(t)$ is unchanged except for having the coordinate   $a_ib_i$ replaced by  $a_ib_i + a_i b_{i+r}$ and  $a_{i+r}b_{i+r}$ replaced by  $a_{i+r}b_{i+r} - a_i b_{i+r}$.
    We may  again use any cost vector {\bf c} $= (c_1,\dots ,c_n)\in \reals^n$  with strictly descending  coordinates $c_1 > c_2 > \cdots > c_n$  to deduce that $u\prec v$ implies $\bc\cdot u < \bc\cdot v$.  
   See  \cite{BW}, \cite{HT}  for further background on the Tamari lattice.
 \end{ex}

%\subsection{Topological consequences for weak order, the Tamari lattice and Cambrian lattices}\label{topol-subsection} %oset topology for these examples}
%In Section ~\ref{application-section}, we will revisit these examples to give 
%applications of our upcoming general theory.  At that 
%point, we  will also put $c$-Cambrian lattices into this framework, with generalized 
%associahedra serving as the associated polytopes in that case.

%Now we recover an assortment of previously known results, all as immediate corollaries of Theorem ~\ref{simple-polytopes}.

\begin{thm}\label{weak-theorem}
%Theorem
%~\ref{simple-polytopes}  implies 
%that e
Each open interval  $(u,v)$ in the weak order has order complex 
which is homotopy equivalent to a ball or a sphere of some dimension.
\end{thm}

\begin{proof}
We obtain the Hasse diagram for weak  order as the 1-skeleton of the permutahedron, which is a simple polytope, using a cost vector as in Example ~\ref{permutahedron-cost-vector}.  A proof that the weak order is  a lattice may be found  in \cite{BB}. 
Thus, Theorem ~\ref{simple-polytopes} applies.
\end{proof}

The homotopy type of the intervals in weak order was previously determined in \cite{Ed}, \cite{EW}, and subsequently by a different method in  \cite{bj4}.

\begin{thm}\label{Tamari-theorem}
%Theorem ~\ref{simple-polytopes} implies that 
Each open interval $(u,v)$ in the Tamari lattice has order complex homotopy equivalent to a  ball or a sphere of some dimension.
\end{thm}

\begin{proof}
The Tamari lattice  has as its Hasse diagram  the 1-skeleton of the associahedron with respect to any cost vector as in Example ~\ref{associahedron-cost-vector}.   A proof that the Tamari lattice  is a lattice appears in \cite{Pallo1}.  
Thus, the Tamari lattice meets all the conditions of Theorem ~\ref{simple-polytopes}.
\end{proof}
 
 The homotopy type of the intervals in the Tamari lattice was previously determined by Bj\"orner and Wachs in \cite{BW}, where they note that this result also essentially follows from work of Pallo in \cite{Pallo}.

Next we combine several results from the literature in a manner suggested to us by Nathan Reading to deduce the following result which generalizes Theorem ~\ref{Tamari-theorem}.

\begin{thm}\label{Cambrian-theorem} 
%Theorem ~\ref{simple-polytopes} implies that e
 Each open interval $(u,v)$   in any $c$-Cambrian lattice has order complex that is homotopy equivalent to a ball or a sphere of some dimension.  
\end{thm}
\begin{proof}
See Proposition 3.1 in \cite{hohlweg} for the fact that the  Hasse diagram of the $c$-Cambrian lattice is obtained from the polytope $Asso_c^a(W)$ known as a generalized associahedron given by $W$ (as defined e.g. in \cite{hohlweg})  
by choosing a suitable cost vector $\bf{c}$  and taking the directed graph that $\bf{c}$  
induces on the 1-skeleton of the polytope.  Just before Example 3.5 in \cite{hohlweg}, it is asserted that all of these polytopes are simple.   This is proven as Theorem 3.4 in \cite{hugh}.  Thus, Theorem ~\ref{simple-polytopes} applies in the case of all $c$-Cambrian lattices.
\end{proof}

Thus, we   recover Reading's results on the homotopy type of intervals in the $c$-Cambrian lattice, thereby showing that all generalized associahedra can  also be handled by our approach.

\subsection{The case of  zonotopes}\label{zonotope-section}

Now we turn to another large class of examples of polytopes to which our  results will apply, namely all simple polytopes which are zonotopes.    Bj\"orner already determined the homotopy type of all open intervals for   zonotopes in \cite{bj4}, but we nonetheless include this discussion so as to show   how  another large and important class  of polytopes fits into our framework.

 \begin{prop}\label{zonotope-non-revisiting}
 Any zonotope $P$ and any  generic linear functional $\bc$ together % give $G(P,\bc )$ which 
 will   satisfy  the nonrevisiting  property, and hence  the Hasse diagram property.  Thus, $G(P,\bc )$ has directed diameter at most 
 %each directed path in $G(P,\bc )$ uses at most $n-d$ edges, 
 $n-d$ for  $n$ the number of facets  in $P$ and $d$  the dimension of $P$.
 \end{prop}
 
 \begin{proof} 
 Any zonotope is a Minkowski sum of line segments.
 Departing a face while increasing the dot product with the cost vector $\bc $ means traversing an edge in the direction of one of these  line segments generating the zonotope.  But we can never traverse an edge going in exactly the opposite direction to this while still increasing the dot product.
By virtue of a zonotope being a Minkowski sum of line segments, it is not 
   possible to return to the face without at some point traversing a parallel edge in the opposite direction.  The proof of 
   Proposition ~\ref{zonotope-lattice}
    %Remark ~\ref{zonotope-from-arrangement} 
    will give another way of seeing why this nonrevisiting  property holds.  
   
   For the last claim, simply observe that each edge in a directed path departs from a facet that may never be revisited and that the final vertex in a directed path will still belong to $d$ facets.  Thus, there can be at most $n-d$ steps since there are at most $n-d$ facets available to be departed at some stage in the directed path.  
 \end{proof}
 
 \begin{prop}\label{zonotope-lattice}
%Next  we show   
%that $G(P,\bc )$ is the Hasse diagram of a  lattice whenever
If  $P$ is a simple polytope that  is a zonotope and $\bc $ is generic, 
then  $G(P,\bc )$ is the Hasse diagram of a lattice.  
\end{prop}

\begin{proof}
We deduce this fact  
by combining assorted known results,  as explained next. % in the next two remarks.
 %
% \begin{rk}\label{zonotope-from-arrangement}
 It is well known that every zonotope may be  obtained from a central hyperplane arrangement as follows.   Any  central hyperplane arrangement induces a subdivision of a unit sphere centered at the origin.  If the arrangement is not essential, then restrict this sphere to a subspace through the origin of as high dimension as possible such that the arrangement restricted to that subspace is essential.   Let  the vertices of the resulting subdivision of the sphere be the vertices of a polytope.  Taking the dual polytope to this, the result is a zonotope, and in fact every zonotope may be realized this way.  The point is to make the hyperplanes perpendicular to the line segments comprising the Minkowski sum of line segments. See e.g. \cite{zi} or \cite{HS}.
   From this perspective, an edge of a zonotope departs a face by crossing one of these hyperplanes, namely one that is perpendicular to the direction of the edge being traversed.  We can never revisit the face we just left without  crossing the hyperplane in the opposite direction.  But this would mean  traversing an edge of the polytope  the opposite direction to the edge we used to depart the face, contradicting  $G(P,\bc )$ being induced by a cost vector $\bc $.    Thus,  for $P$ a simple zonotope and $\bc $ generic, this implies that 
    $G(P,\bc )$ must satisfy the nonrevisiting property (so in particular must be a Hasse diagram).   
% \end{rk}
%
%Remark ~\ref{zonotope-from-arrangement} has  the following consequence:
%
%We will prove shortly  that the directed graph  $G(P,\bc )$ which a generic cost vector $\bc$ induces on the 1-skeleton of a  zonotope $P$ will always be a Hasse diagram.  Our method will be  first to prove that $G(P,\bc )$ always has the non-revisiting property. 
%

 %\begin{rk}\label{zonotope-lattice}
It is proven in \cite{BEZ}  that the poset of regions given by a central, simplicial hyperplane arrangement is a lattice.  Given a simple zonotope $P$ and generic cost vector $\bc$, the poset having  $G(P,\bc )$ as its Hasse diagram 
is exactly the poset of regions of a central, simplicial hyperplane 
arrangement, hence is always a lattice.  
% Remark ~\ref{zonotope-from-arrangement} explains the connection between zonotopes and central hyperplane arrangements that this translation of the lattice  result from ~\cite{BEZ}  relies upon. 
% \end{rk}
\end{proof}
 
\begin{thm}\label{zonotope-homotopy}
Whenever a zonotope is a simple polytope, then the poset given by $G(P,\bc)$  has each open interval  homotopy equivalent to a ball or a sphere.
\end{thm}

\begin{proof}
The first thing to note is that the poset will always be a lattice in this case, by Proposition ~\ref{zonotope-lattice}.
% the results mentioned in  Remark ~\ref{zonotope-lattice}.   
 The Hasse diagram property is proven for all zonotopes  in Proposition ~\ref{zonotope-non-revisiting}. 
 Thus,  Theorem ~\ref{simple-polytopes} applies to all simple  zonotopes.  
\end{proof}

\subsection{%More general directed graph structures on 1-skeleta of simple polytopes: p
More general facial orientations of  simple regular CW spheres}
%shellings of simplicial %dual polytopes to simple 
%polytopes}
\label{posets-from-shellings}

Vic Reiner raised the question  (personal communication) of whether we could use 
 Proposition 5.3  from  \cite{AER}, a result that is recalled  as Proposition ~\ref{AER-result} below,   to generalize our 
results.  Specifically, he suggested generalizing   from our framework of  acyclic 
orientations on 1-skeleta of simple polytopes  given by cost vectors %derived from line shellings of their dual polytopes 
  to   more general  acyclic orientations known as facial orientations; Reiner also suggested  trying to prove results  
  for a somewhat larger  class of   regular CW  spheres than just the simple polytopes discussed so far. 
  
   Recall that a {\bf facial orientation}  of the 1-skeleton of a regular CW complex  $K$ 
 is an orientation  $\mathcal{O}$ of the 
1-skeleton graph of  $K$ %the regular CW complex 
such that for each cell $\sigma  \in K$, the restriction of 
$\mathcal{O}$ 
%this  oriented  1-skeleton 
to the closure of $\sigma $ has a unique source and a unique sink. 
  % with arbitrary  facial orientations on their 1-skeleta.
    It is well known 
  % was shown in \cite{ }
  that a % We discuss shortly how a 
shelling of a simplicial polytope is equivalent to a facial orientation of its dual polytope; the special case of line shellings is also
well-known to yield  precisely those facial orientations which are induced by cost vectors.  One may easily construct 
examples demonstrating that  not all 
 facial orientations can be  induced by cost vectors.  
%Next  we recall an extension   in  \cite{AER}  to regular CW 
%spheres whose duals are simplicial spheres, motivating Reiner's question.  
%We will explain 
%In Remark ~\ref{limited}, we  will explain  how the  proofs we give for  our main results  
%do not generalize to the extent proposed by Vic Reiner.  However, we do obtain the following generalization,  going part of the way towards what Reiner proposed.

\begin{prop}[Proposition 5.3 of \cite{AER}]\label{AER-result}
Let $X$ be  a  shellable regular CW sphere with  $P$ its  face poset.  There is a dual regular CW sphere, denoted  $X^*$,  with face poset $P^*$.  Letting  $G(P^*)$ denote  the graph arising as the   1-skeleton of $X^*$, then the acyclic orientation 
$\mathcal{O} $  of $G(P^*)$ induced by any  shelling order of $X$ % a regular 
%CW sphere with face poset $P$ 
is  a facial orientation on the graph of $X^*$.
\end{prop}

Our techniques  do yield the following partial  answer to  Reiner's question.   
On the other hand, Example ~\ref{cannot-generalize}  in conjunction with Remark ~\ref{limited}
constrains the extent to which a positive answer 
to Reiner's question is possible.
%We  will also soon give an  example  limiting how positive an answer  is possible.  

\begin{thm}\label{more-general-results}
Let $P$ be a simple polytope, and let $\mathcal{O}$ be a facial acyclic 
orientation on its 1-skeleton.  Suppose that
the directed graph on the 1-skeleton of $P$ induced by $\mathcal{O}$ 
is the Hasse diagram of a lattice $L$.  Then $L$ has the following properties:
\begin{enumerate}
\item
The pseudo-join of any collection $\{ a_1, a_2, \dots ,a_i \} $  of elements of $L$ 
all covering a common element $u$ will equal the join 
% the join of this same collection of elements, namely will equal 
 $a_1\vee a_2\vee \cdots \vee a_i $ of these same elements.
\item
For $S, T$ distinct collections of elements all covering a fixed element $u$, then the pseudo-join of the elements of $S$ will not equal the pseudo-join of the elements of $T$.
\item
Each open interval in $L$ has order complex homotopy equivalent to a ball or a sphere.
\end{enumerate}
\end{thm}

\begin{proof}
Statements (1), (2) and (3), respectively, were already proven for 
$\mathcal{O} $ induced by a cost vector  
%will indeed also apply more generally for facial orientations  without need for any modification.   These earlier results to be generalized are 
in Theorem ~\ref{pseudo=join}, Lemma ~\ref{new-pseudo-joins} and Theorem ~\ref{simple-polytopes}, respectively.  The  same proofs still hold entirely unchanged for more general facial orientations.  Checking this is left as a completely straightforward exercise for the interested reader.  
\end{proof}

\begin{rk}\label{limited}
%Proposition 5.3 in \cite{AER} proves that the acyclic orientation induced on the 
%1-skeleton of a regular 
%CW sphere by a recursive coatom order on the dual poset to its face poset will have the 
%property that each face has a unique source and a unique sink. 
%%% Indeed we may replace our polytope 
%%%$P$ by such a regular CW sphere and our cost vector $\bc $ by such an acyclic orientation. 
%One might hope to  extend  our results  to  this more general context.  
%%%%To do so with proofs similar to those that we have given, we would  certainly also need to assume the %natural analogue for regular CW complexes  of having  a simple polytope as well as needing  %to  assume that our directed graph on the 
%%%the oriented 1-skeleton to have the natural analogue of the Hasse diagram property.    
%%%
%%%A more serious seeming limitation  is that 
%%%However, o
Our proof of Theorem
~\ref{2-atom-base-case} relies in an essential way on the property of polytopes
that two distinct  2-dimensional faces cannot share both an edge and a  vertex not in that edge.  Our proofs of 
Theorem ~\ref{pseudo=join}, Lemma ~\ref{new-pseudo-joins} and Theorem ~\ref{simple-polytopes} all 
rely upon Theorem ~\ref{2-atom-base-case}.
This property  of polytopes  used in the proof of  Theorem ~\ref{2-atom-base-case}
%of 2-dimensional faces in polytopes which is used in a seemingly essential way in
 %the proof of Theorem  ~\ref{2-atom-base-case} 
 does not hold for regular CW spheres in general, 
even with the further assumption that the  regular CW sphere 
is simple.  Example ~\ref{cannot-generalize} exhibits  this non-implication, showing that Theorem ~\ref{more-general-results} cannot be extended from polytopes to regular CW spheres.
\end{rk}

\begin{ex}\label{cannot-generalize}
Now we will  construct a simple regular CW sphere with two 2-cells sharing an edge and also  sharing a vertex that is disjoint from that edge.  To this end, we give a regular CW decomposition of the boundary of a cylinder as follows.  
Begin by placing four vertices denoted $v_1,v_2,v_3,v_4$ clockwise about the boundary of the upper disk comprising the top of the  cylinder.  Now likewise put four vertices denoted $w_1,w_2,w_3,w_4$ clockwise about the boundary of the  bottom disk comprising the bottom of the cylinder.  Introduce edges $e_{v_i,v_j}$ for each $i\ne j$ other than the pair $i=1,j=3$.  Likewise introduce edges 
$e_{w_i,w_j}$ for each $i\ne j$ other than the pair $i=1,j=3$.  Also introduce edges $e_{v_1,w_1}$ and $e_{v_3,w_3}$.  The resulting subdivision of this 2-sphere, namely of the boundary of a cylinder, will also have the following six 2-cells.  There are 2-cells  with vertices $\{ v_1,v_2,v_4 \} $ and with $\{v_3,v_2,v_4\} $ covering the top disk, 2-cells with vertices 
$\{ w_1,w_2,w_4\} $ and with $\{ w_3,w_2,w_4\} $ covering the bottom disk,  and 2-cells $F$ and $F'$  with vertex sets
$\{ v_1,v_2,v_3,w_1,w_2,w_3 \} $ and with $\{ v_1,v_4,v_3, w_1, w_4,w_3 \} $ covering the remainder of the boundary of the cylinder.  The  2-cells $F$ and $F'$ share the edge 
$e_{v_1,w_1} $ and also  the edge $e_{v_3,w_3}$.  Thus, $F$ and $F'$  share four  vertices with  two of these four  vertices comprising an edge.  
\end{ex}

\section{Further questions and remarks}\label{further-section}

\begin{rk}
In \cite{Gr}, Curtis Greene raised the question of finding interesting classes of posets with each open interval having M\"obius function $0, 1, $ or $-1$.  Theorems ~\ref{simple-polytopes}
and ~\ref{more-general-results}
speak to that question by giving large classes of such posets.
\end{rk}

\begin{rk}
 We refer readers to   \cite{Kalai} for interesting, related  work  that takes a  somewhat similar perspective to ours, work which  provided  some inspiration for parts of our work.  In \cite{Kalai}, Kalai 
 proved  that the combinatorial type of a simple polytope is determined by its 1-skeleton,  also making use of  a cost vector in this construction as well as utilizing   the consequent  sources and sinks of the various faces. 
 % More specifically, Kalai shows how to determine its faces and which set of vertices is contained in each face.
\end{rk}

%\begin{rk}\label{simple-enough}
%Klee and Walkup observed in \cite{KW} that to prove the Hirsch Conjecture, it would suffice to prove it for simple polytopes.  It is possible that the hypothesis in our results that the  polytopes we 
%consider are simple is not actually necessary for the results to hold.  It is certainly the case that the proofs that  we  have given do rely heavily upon this hypothesis.
%\end{rk}

One might be tempted, in light of our results,  to ask  the following question: 

\begin{qn}\label{tempting-question}
Let  $P$ be  a simple polytope and  let $\bf{c}$ be a generic cost vector  such that $G(P,\bf{c})$ is the Hasse diagram for a poset.  Does this imply that this poset is a lattice?
\end{qn}

An affirmative answer would  have allowed our hypotheses throughout much of this paper 
 to be relaxed from lattice to poset. %, a much more readily checkable condition.  
However, 
Francisco Santos has  provided the following example, showing  that the  
answer to Question ~\ref{tempting-question} is negative in general. 

\begin{ex}[Francisco Santos]
  Start with an octahedron $P$ with two antipodal vertices as source and sink, leaving four intermediate vertices  $v_1,v_2,v_3,v_4$ connected with each other with the structure of a  4-cycle.  Put two opposite vertices  $v_1,v_3$ among these four vertices  at a higher height than the other two, namely with $\bc \cdot v_i > \bc \cdot v_j $ for each $i\in \{ 1, 3\} $ and each $j\in \{ 2, 4\} $.
  %and $\bc  \cdot v_3 > \bc \cdot v_2$ and $\bc \cdot v_1 > \bc \cdot v_4$ and $\bc \cdot v_3 > \bc \cdot v_4$. 
    %That is, we let $\bc $ be a vertical vector so we may speak higher vertices having higher cost. 
     Now truncate each of the six  vertices by slicing  by a generic hyperplane with slope chosen so as  to make this a simple polytope with the Hasse diagram property (with each of the original vertices replaced by four new vertices).  This will yield a simple polytope with $G(P,\bc )$ a Hasse diagram for a  poset that is not a lattice, since there will be a pair of vertices having two different least upper bounds; specifically, we may use one of the four vertices replacing $v_2$ together with one of the four vertices replacing $v_4$.  We may choose such vertices so that we get one 
  least upper bound in the quadrilateral replacing $v_1$ and another in the quadrilateral replacing $v_3$.
  \end{ex}

One may also  construct a 
polytope $P$ and a cost vector $\bc$  such that $G(P,\bf{c})$ is the Hasse 
diagram of a poset  with the pseudo-join of some collection of atoms which is
not equal to  the  least upper bound of this same collection of 
atoms, as shown next.   

\begin{ex}\label{not-join=pseudo}
Start with a $3$-dimensional cube and  add a new vertex by coning over  one of the facets of the cube that contains the vertex of the cube where the cost vector was maximized, positioning this new vertex so that  it becomes the pseudo-join
of all the atoms.  This can be done by letting ${\bf c} = (100, 2, 1)$, letting the vertices of the cube be
$(\pm 1, \pm 1, \pm 1)$ and taking as the cone point over a facet of this cube the vertex 
$(2, 0, 0)$.  
To make  the 1-skeleton of the polytope obtained this way  a Hasse diagram, we  cut off the vertex
$(2,0,0)$  of the cone with a hyperplane near this vertex with a slope for this slicing hyperplane chosen  in such a way  that one of the resulting four new vertices (replacing $(2,0,0)$)  becomes the pseudo-join of all the atoms,  while the vertex that was the least upper bound of the atoms in the original cube still remains the least upper bound  of all the atoms.  

This is not a simple polytope,  but 
one may transform this into a simple polytope  by shaving by a hyperplane  at each node of degree higher than 3.  However, that shaving operation will change which element is the join of the set of  three atoms in such a way  that indeed the  join of the three atoms is  the pseudo-join of the same 
set of three atoms, transforming  this into a positive example of our result that joins equal pseudo-joins for simple polytopes.
\end{ex}

To make our results more effective on naturally arising examples, it could also help to answer the following question:

\begin{qn}\label{good-way-recognize}
Is there a good way to recognize when $G(P,\bc )$ will be the Hasse diagram of a poset?  Is there an effective  way to determine when this  poset will be a lattice?
\end{qn}

Regarding the first part part of Question ~\ref{good-way-recognize}, Louis Billera has suggested considering the directed adjacency matrix $A$ where he observed that the Hasse diagram property would imply that the  trace of $A^T \cdot A^i$ would need to be 0 for each $i\ge 2$, letting $A^T$ denote the transpose of $A$.  From the standpoint of algorithmic efficiency, this requires an $n\times n$  matrix where $n$ is the number of vertices of $G(P,\bc )$, which may be much larger than either  the dimension or the number of facets in the polytope.

 \begin{rk} 
  It may seem natural now  to ask 
      whether a  simple polytope $P$ together with a 
      generic cost vector $\bf{c}$ such that $G(P,\bf{c})$ is the Hasse diagram of a poset will always satisfy the directed graph version of the  nonrevisiting path conjecture.  An obvious place to start is  to ask  whether any of the known counterexamples to the Hirsch Conjecture give rise to directed graphs $G(P ,\bc ) $  that 
meet the  hypotheses of Conjecture ~\ref{main-conj} or at least are Hasse diagrams of posets, since 
these polytopes  are all known to be counterexamples to the undirected version of the nonrevisiting path conjecture.  
Lemma ~\ref{source-non-revisit} implied  that there are no counterexamples of this type  to Conjecture ~\ref{main-conj}. 
 
The original construction of Francisco Santos in \cite{Sa} of a polytope violating the Hirsch Conjecture 
 was a spindle (see Definition ~\ref{spindle-def}) but was  not a simple polytope.  His presentation for this polytope is essentially as  an $H$-polytope, in that he gives the vertices of its dual polytope (from which the bounding hyperplanes of the original polytope may easily be deduced).  Santos  remarks on p. 389 in \cite{Sa} that determining the vertices of this polytope seems computationally out of reach, which we note also makes determining the undirected graph of the 1-skeleton elusive.  In particular, this makes  the directed graph, $G(P,\bc )$ for any particular choice of $\bc $, also computationally out of reach.  A second type of  computational challenge to thoroughly examining these examples 
 would be the need to consider all possible generic cost vectors $\bc $.  There are exponentially many orientations on  the 1-skeleton graph  to consider (as a function of the number of graph edges), though in principle one would only need to consider the  ``good orientations'' in the sense of Kalai from \cite{Kalai}.    Thus, there are multiple substantial challenges to understanding this example in full, but in any case it will not yield a counterexample to Conjecture ~\ref{main-conj}.  
 
 The later smaller counterexamples of Matschke, Santos, and Weibel to the Hirsch Conjecture appearing  in \cite{MSW}  are  simple polytopes that are spindles.    
 All of these known examples of polytopes   violating the Hirsch Conjecture result from  $d$-polytopes which are spindles  with $n$ facets having the property that the known  pair of vertices at distance greater than $n-d$ from each other are  the two  distinguished vertices in the  spindle.  
 Our  Theorem \ref{santos-okay}  shows  for these examples where  $P$ is a simple  spindle  
 that $G(P,\bc )$  is not a Hasse diagram.  Thus, 
Theorem ~\ref{santos-okay} shows that these constructions violating the Hirsch Conjecture  do not also serve as counterexamples to our Conjecture ~\ref{main-conj}.
 \end{rk}

\begin{rk}
In seeking more examples of polytopes fitting into  our framework, 
 one might be  tempted  to  consider   fiber polytopes (introduced %,  namely the
 %polytopes introduced 
 in \cite{BS}); after all, the  permutahedron and associahedron  are both fiber polytopes and more specifically are monotone path polytopes, and both  the permutahedron and associahedron  do 
 fit into our framework.  However,  every polytope $P$ may be realized 
 as a monotone path polytope as follows. 
  Take  the join of $P$  with a point  $p$.  Project the resulting polytope  $P'$  to the real  
 line by a linear map $\pi $  in such a way that 
 the fiber  $\pi^{-1}(t)$ over each point $t$  on the real line is either empty, the single point  $p$, or has the combinatorial type of $P$.
 One may check that   the fiber polytope resulting from the map  $\pi : P'\rightarrow \reals $ has the combinatorial type of $P$. 
 Thus, monotone path  polytopes are too general a class of polytopes to hope for our results to apply to all of them.
  
% Another natural seeming  class of polytopes to try, 
 Focusing on generalized permutahedra (see \cite{Post}) still gives a class that 
 will not always satisfy all of our hypotheses. 
 % This is in spite of multiple  important 
 %classes of generalized permutahedra (e.g. traditional permutahedra and associahedra) fitting into
 %our framework.
   To see  this, note  that generalized permutahedra 
 sometimes have  triangular  faces, forcing   the Hasse diagram property to fail  for
 all choices of cost vector.
\end{rk}

We conclude with one of the most  natural and potentially important questions that stems from our work:      %that we have so far only considered to a limited extent 

\begin{qn}\label{applic-in-our-setting}
Do any of the  important classes of 
polytopes coming from real-world problems studied in  
operations research fit into our framework?  In other words, are they simple polytopes whose 1-skeleta (with respect to a choice of cost vector)  are Hasse diagrams of lattices?   One can either ask this for a fixed cost vector, or preferably for all generic cost vectors for the given simple polytope.
% Ideally, one would hope for this for all generic cost vectors for a fixed simple polytope, but one could also ask which cost vectors yield these properties. 
%One could  ask this for all generic cost  vectors or for a given cost vector of interest.  
\end{qn}

Given an affirmative answer for some such class of polytopes, Conjecture ~\ref{main-conj} would assert % for such polytopes  
that  the simplex method for linear programming  applied to this class of polytopes  would  require at most $n-d$ steps  regardless of choice of pivot rule.    It is plausible that not all of the hypotheses in Conjecture ~\ref{main-conj} are necessary, though they all seem like they would be helpful for proving Conjecture ~\ref{main-conj}; if a version of Conjecture ~\ref{main-conj} with fewer hypotheses  could be proven, this would of course increase the potential for an affirmative answer to Question ~\ref{applic-in-our-setting}.
% e.g. with the lattice hypothesis removed, that would increase the chances  that real-world polytopes arising in operations research could be proven to have the face non-revisiting property.  
% at least when one uses a 
%generic cost vector? 
%, at least for the types of cost vectors that typically arise in  real world applications?
%\end{qn}
%

%In terms of candidates for Question ~\ref{applic-in-our-setting}, w
We have ruled out some 
classes of real-world polytopes as  potential 
candidates for Question ~\ref{applic-in-our-setting}.  Specifically, we have found small examples of  transportation polytopes, 0/1-polytopes, vertex decomposable polytopes, 
and  traveling salesman polytopes that do not meet all of the hypotheses for Conjecture ~\ref{main-conj} in its present form by virtue of having triangular faces.  
However,  many more  classes arising in operations research remain.  

An affirmative answer to Question ~\ref{applic-in-our-setting} could  contribute a new perspective for particular classes of polytopes   to  the  understanding of 
%help explain (in such cases) 
 the widely observed phenomenon 
that the simplex method typically  is much  more efficient 
in practice in real-world applications  than is predicted by worst case analysis.  
The simplex method is already well known  to have polynomial  time average case complexity in many cases;  
%While none of these average case analysis results could predict the efficiency in many real-world situations that involve inputs that are far from the situation considered in average case analysis, 
%moreover, 
Spielman and Teng proved the much stronger result 
that the simplex method has so-called polynomial smoothed complexity (see  \cite{SpT}), i.e. has polynomial complexity in cases that interpolate  between average case and worst case.  % This better captures real-world behavior, thereby largely explaining the observed phenomenon.   

 One way to view  much of our work in this paper is as an attempt  to better understand conceptually what properties of  a polytope $P$ and cost vector $\bc $ will 
 suffice to ensure that directed paths in $G(P,\bc )$ may never revisit any faces that they have left, thereby also ensuring for such $P$ and $\bc $ that  
 the simplex method is efficient even in the worst case, i.e.  no matter what  choices of  pivots are made.   Others such as in  ~\cite{BDL}  also have recently examined (and made headway) on  the closely related question of proving  that particular  classes of polytopes and cost vectors arising naturally in operations research   will make linear programming efficient for all choices of pivot rule when certain parameters are held fixed.
 %; rather than focusing on conceptual properties a polytope may have such as its 1-skeleton being a Hasse diagram of a poset, the focus in \cite{BDL}  is more on taking on particular real-world classes of polytopes, a very interesting question indeed. 
  There is certainly much more to be done in this area, including the question of resolving our main  conjecture which in spite of the result  described  in the upcoming  appendix  does still remain an open question.

\section{Appendix by Dominik Preu\ss : Monotone Hirsch conjecture for simple polytopes whose 1-skeleta are Hasse diagrams of lattices}
%\title{Appendix: Monotone Hirsch conjecture for simple polytopes whose 1-skeleta are Hasse diagrams of lattices}
%\author{Dominik Preu\ss \\dom.preuss@web.de}
%\theoremstyle{plain}
%\newtheorem{conjecture}{Conjecture}
%\theoremstyle{definition}
%\newtheorem*{definition}{Definition}
%\theoremstyle{plain}
%\newtheorem{obs}{Observation}
%\theoremstyle{plain}
%\newtheorem{lemma}{Lemma}
%\theoremstyle{plain}
%\newtheorem*{theorem}{Theorem}
%\begin{document}
%\maketitle
Here we prove that if $G(P,\mathbf{c})$ is the Hasse diagram of a lattice, any directed path in $G(P,\mathbf{c})$ has length at most $n-d$. In particular, $P$ satisfies the monotone Hirsch conjecture.

To that end, recall from lattice theory that an element $b$ of a lattice $L$ is called \emph{join-irreducible} iff $b = \bigvee_{a \in S} a$ for a set $S \subseteq L$ implies $b \in S$. It makes sense to define the join of the empty set as the unique minimal element of the lattice, hence that element is not join-irreducible by definition. We shall need the following property of join-irreducibles:
\begin{lem}[\cite{st}, Chapter 3.4]
An element of a finite lattice is join-irreducible iff it covers precisely one element.
\end{lem}
As an immediate consequence of Lemma 1, we have:
\begin{lem}
Let $P$ be a simple polytope and $\mathbf{c}$ a generic cost vector such that $G(P,\mathbf{c})$ is the Hasse diagram of a lattice. Then a vertex $v$ is join-irreducible iff it is distinct from $\hat{0}$ \emph{and} is the unique source of a facet.
\end{lem}
\begin{proof}
If $v$ is not equal to $\hat{0}$, then there must be a directed path from $\hat{0}$ to $v$, since $\hat{0}$ is a source of the whole graph. Hence there must be an edge which ends in $v$, which means that $v$ covers at least one other lattice element. On the other hand, if $v$ is the source of a facet $F$, then it must be the initial vertex of every incident edge that belongs to $F$. There are $d-1$ such edges, where $d$ is the dimension of $P$ (since $P$ is simple). Since $v$ is incident to a total of $d$ edges (again, since $P$ is simple), only one more edge remains that can end in $v$. Since there is at least one such edge, there is \emph{precisely} one such edge. Hence $v$ covers precisely one lattice element, and therefore, it is join-irreducible by Lemma 7.1.

Conversely, suppose $v$ is join-irreducible. Then it covers precisely one lattice element/vertex $u$. Therefore, except for the edge $\{u,v\}$, all edges incident to $v$ point away from $v$. Since $P$ is simple, there is a facet $F$ for every set $X$ of $d-1$ edges incident to $v$ such that $X$ are precisely those edges incident to $v$ that lie in $F$. Therefore, there is a facet in which all edges point away from $v$, implying that $v$ is the source of that facet.
\end{proof}
For the proof of our main theorem, we need one further purely lattice-theoretic result:
\begin{lem}
If $b,c$ are elements of a finite lattice  $L$ such that $b$ covers $c$, then there is a join-irreducible $j \in L$ such that $c \vee j = b$ (where $j$ is not necessarily distinct from $b$).
\end{lem}
\begin{proof}
Since every element of a finite lattice is a join of join-irreducibles, we may choose a minimal set $S$ of join-irreducibles such that $b = c \vee \bigvee_{a \in S} a$. If $S$ contains only one join-irreducible, we are done. Otherwise, choose $j \in S$ and consider the element $c \vee \bigvee_{a \in S \setminus \{j\}} a$. This must be $\geq c$ and $\leq b$, which means that it is either equal to $b$ or to $c$, since $b$ covers $c$. If it were equal to $b$, then this would contradict the minimality of $S$. If it were equal to $c$, then $c \vee j = c \vee (\bigvee_{a \in S \setminus \{j\}} a) \vee j = c \vee \bigvee_{a \in S} a = b$, and so $\{j\}$ would also be a proper subset of $S$ whose join with $c$ is $b$, again contradicting the minimality of $S$.
\end{proof}
Our main theorem now is:
\begin{thm}
Let $P$ be a simple polytope and $\mathbf{c}$ a generic cost vector such that $G(P,\mathbf{c})$ is the Hasse diagram of a lattice. Then every directed path in $G(P,\mathbf{c})$ has length at most $n-d$, where $n$ is the number of facets and $d$ the dimension.
\end{thm}
\begin{proof}
Suppose we are given a directed path consisting of vertices $v_0,\dots,v_m$. Then by Lemma 7.3, there is a join-irreducible $j_i$ for every $i > 0$ such that $v_i = v_{i-1} \vee j_i$. In other words, $v_i = v_0 \vee \bigvee_{k=1}^i j_k$. It follows that all the $j_k$ must be distinct, since $v_i = v_0 \vee \cdots \vee j_k \vee \cdots \vee j_{i-1} \vee j_i$ and $j_k = j_i$ imply together with the idempotence of the join operation that \begin{multline*} v_i = v_0 \vee \cdots \vee j_k \vee \cdots \vee j_{i-1} \vee j_i \\ = v_0 \vee \cdots \vee j_k \vee \cdots \vee j_{i-1} \vee j_k \\ = v_0 \vee \cdots \vee j_k \vee \cdots \vee j_{i-1} = v_{i-1}\end{multline*}
Therefore the length of a path is at most the number of join-irreducibles. But by Lemma 7.2, there is precisely one join-irreducible for every facet \emph{except} the $d$ facets whose source is $\hat{0}$, so the number of join-irreducibles is $n-d$.
\end{proof}

\end{document}